\documentclass[12pt]{article}
\usepackage{amssymb,amsthm,amsmath,amsfonts,latexsym,tikz,hyperref,start4,shuffle}
\usepackage[hmargin=.7in,vmargin=1in]{geometry}
\usetikzlibrary{arrows.meta}
\allowdisplaybreaks

\newtheorem{thm}{Theorem}[section]
\newtheorem{prop}[thm]{Proposition}
\newtheorem{cor}[thm]{Corollary}
\newtheorem{lem}[thm]{Lemma}

\theoremstyle{definition}
\newtheorem{definition}[thm]{Definition}
\newtheorem{remark}[thm]{Remark}
\newtheorem{example}[thm]{Example}

\DeclareMathOperator{\cDes}{cDes}
\DeclareMathOperator{\Pk}{Pk}
\DeclareMathOperator{\pk}{pk}
\DeclareMathOperator{\cPk}{cPk}
\DeclareMathOperator{\cpk}{cpk}
\DeclareMathOperator{\tor}{tor}
\DeclareMathOperator{\cyc}{cyc}
\DeclareMathOperator{\cQSym}{cQSym}
\DeclareMathOperator{\QSym}{QSym}
\DeclareMathOperator{\sym}{Sym}
\DeclareMathOperator{\Comp}{Comp}
\DeclareMathOperator{\cComp}{cComp}

\numberwithin{equation}{section}

\begin{document}
\pagestyle{plain}

\title{Enriched toric $[\vec{D}]$-partitions}

\author{Jinting Liang\\[-5pt]
\small Department of Mathematics, Michigan State University,\\[-5pt]
\small East Lansing, MI 48824-1027, USA, {\tt liangj26@msu.edu}}
\date{\today\\[10pt]
	\begin{flushleft}
	\end{flushleft}}

\maketitle

\begin{abstract}
This paper develops the theory of enriched toric $[\vec{D}]$-partitions. Whereas Stembridge's enriched $P$-partitions give rises to the peak algebra which is a subring of the ring of quasi-symmetric functions $\QSym$, our enriched toric $[\vec{D}]$-partitions will generate the cyclic peak algebra which is a subring of cyclic quasi-symmetric functions $\cQSym$.  In the same manner as the peak set of linear permutations appears when considering enriched $P$-partitions, the cyclic peak set of cyclic permutations plays an important role in our theory. The associated order polynomial is discussed based on this framework.
\end{abstract}

\section{Introduction} 
Denote by $\mathbb{N}$ and $\mathbb{P}$ the set of nonnegative integers
and positive integers respectively.
For $m,n\in\mathbb{N}$, define $[m,n]=\{m,m+1,\ldots,n\}$ and, as an abbreviation, we write $[n]=[1,n]$ when $m=1$. A {\em linear permutation} of a set $A\subset \mathbb{P}$ 
is an arrangement $w=w_1w_2\dots w_n$ of elements in $A$ where each element is used exactly once. 
In this case, we call $n$ the {\em length} of $w$, written as $\#w=|w|=n$. Let $\mathcal{S}_n$ be the symmetric group on $[n]$ viewed as linear permutations of $[n]$.

A {\em linear permutation statistic} is a function whose domain is the set of all linear permutations. For a linear permutation $w=w_1w_2\dots w_n$, a {\em descent} of $w$ is a position $i$
    such that $w_i>w_{i+1}$. The {\em descent set} $\Des$ is defined by
    \[
    \Des w=\{i\mid \text{$i$ is a descent of $w$}\}\sbe[n-1].
    \]
    The {\em descent number} of $w$ is $\des w:=|\Des w|$. A {\em peak} of $w$ is a position $i$ such that $w_{i-1}<w_i>w_{i+1}$. The {\em peak set} $\Pk$ is defined by
    \[
    \Pk w=\{i\mid \text{$i$ is a peak of $w$}\}\sbe[2,n-1].
    \]
    The {\em peak number} is $\pk w:=|\Pk w|$.

Quasi-symmetric functions first appeared implicitly as generating functions in Richard Stanley's theory of $P$-partitions~\cite{MR0332509}, 
and then were studied by Ira M. Gessel~\cite{MR777705} in his work. To be more precise, for a finite poset $P$, the set of $P$-partitions can be partitioned by the linear extensions of $P$, while each subset corresponds to a fundamental quasi-symmetric function indexed by the descent set $\Des$ of that linear permutation. The ring $\QSym$ of quasi-symmetric functions was further developed, see~\cite{MR1982883,MR3810249,MR1897637} for partial related articles.   
It also found applications in enumerative combinatorics, representation theory and algebraic geometry~\cite{MR1794711,MR2196760,MR3203651}. In the same vein, Stembridges's work~\cite{S:epp} on enriched $P$-partitions gave rise to the algebra of peaks $\Pi$, a graded subring of $\QSym$, which is closed related to the peak set $\Pk$.

For a linear permutation $w=w_1w_2\dots w_n$, we define the corresponding {\em cyclic permutation} $[w]$ to be the set of rotations of $w$, that is,
\[
[w]=\{w_1w_2\dots w_n, \hs{7pt} w_2\dots w_nw_1,\hs{7pt}\ldots,\hs{7pt} w_nw_1\dots w_{n-1}\}.
\]
Let $[\mathcal{S}_n]$ denote the set of cyclic permutations on $[n]$. 

The {\em cyclic descent set} $\cDes$ of a linear permutation $w$ is defined by
    \[
    \cDes w=\{i\mid \text{$w_i>w_{i+1}$ where the subscripts are taken modulo $n$} \}\sbe[n].
    \]
This leads to the {\em cyclic descent set} of a cyclic permutation
     \[
    \cDes[w]=\{\{ \cDes\sigma\mid\sigma\in[w] \}\},
    \]
    where the double curly brackets denote a multiset. The multiplicity comes into play since $\cDes$ may have the same value on different representatives $\sigma$ in $[w]$. In fact, one can regard the cyclic permutation statistic $\cDes$ as an analogy of $\Des$ in the linear setting. Similarly, if we define the {\em cyclic peak set} $\cPk$ of a linear permutation $w$ by
    \[
    \cPk w=\{i\mid \text{$w_{i-1}<w_i>w_{i+1}$ where the subscripts are taken modulo $n$} \}\sbe[n],
    \]
    then the cyclic counterpart of $\Pk$, the {\em cyclic peak set} $\cPk$ of a cyclic permutation is defined as
     \[
    \cPk[w]=\{\{ \cPk\sigma\mid\sigma\in[w] \}\}.
    \]
    
    \begin{example}\label{ex:set}
Consider permutation $w=3124$ on $[4]$, then 
\[
\Des w=\{1\},\;\cDes w=\{1,4\},\;\Pk w=\emptyset,\;\cPk w=\{4\}.
\]
\end{example}

  \begin{remark}\label{remark: cPk[w]}
         By definition, $\cDes[w]$ carries the information for all representatives in $[w]$. Moreover, $\cDes w$ together with $|w|$ will be sufficient to determine $\cDes[w]$. 
        In fact, $\cDes[w]$ is simply collecting all cyclic shifts of $\cDes w$ in $[n]$ where $n=|w|$, namely,
        \[
        \cDes[w]=\{\{\,i+\cDes w\mid i\in[n] \,\}\}.
        \]
        Here $i+\cDes w$ is the set defined by (\ref{eq:i+E}).
        Similarly, $\cPk[w]$ can be entirely determined by $\cPk w$ and $|w|$.

    \end{remark}

\begin{example}
Consider the permutation $w=1423$ and the corresponding cyclic permutation $[w]=\{1423,3142,2314,4231\}$, we have 
\[
\cDes [w]=\cPk[w]=\{\{\,\{2,4\},\{1,3\},\{2,4\},\{1,3\}\,\}\}.
\]
Noting that $\cDes w=\cPk w=\{2,4\}$, one can easily check the previous remark does hold.
\end{example} 
    
In the work~\cite{agrr:cqf} of Adin, Gessel, Reiner, and Roichman, the ring $\cQSym$ of cyclic quasi-symmetric functions was introduced from toric $P$-partition enumerators, in which case the cyclic descent set $\cDes$ plays an important role. The authors also asked for a cyclic version of the algebra of peaks, to which question we will give an answer in this paper.

This article is devoted to the study of enriched toric $[\vec{D}]$-partitions. By the end, we will construct an algebra of cyclic peaks in $\cQSym$ analogously as the algebra of peaks. The rest of this paper is structured as follows. In the next section, we recall definitions of various terms such as quasi-symmetric functions and cyclic quasi-symmetric functions, with several concrete examples provided. Section~\ref{ep} introduces enriched $\vec{D}$-partitions in terms of directed acyclic graphs (DAGs). In section~\ref{etp}, we define enriched toric $[\vec{D}]$-partitions and develop some of their properties. Section~\ref{we} will review the weight enumerators of enriched $\vec{D}$-partitions from Stembridge and discuss the cyclic analogies for enriched toric $[\vec{D}]$-partitions. Moreover, the weight enumerators corresponding to different cyclic peak sets generate a subring of $\cQSym$ and we call it the algebra of cyclic peaks. We also compute the order polynomial of enriched toric $[\vec{D}]$-partitions.

\section{Basic definitions and results}\label{bd}

\subsection{Sets and compositions}

 We will use $\leq$ with no subscript to denote the ordinary total order on $\mathbb{Z}$, the set of integers. For $n\in\mathbb{P}$, let $2^{[n]}$ denote the set of all subsets of $[n]$, and $2^{[n]}_0$ be the set of all nonempty subsets of $[n]$. Denote by $\Comp_n$ the set of all compositions of $n$ and write $\alpha\vDash n$ for $\alpha\in\Comp_n$. Define a {\em cyclic shift} of a subset $E\sbe[n]$ in $[n]$ to be a set of the form 
\begin{equation}\label{eq:i+E}
    i+E=\{\text{$i+e$ (mod n)}\mid e\in E\}.
\end{equation}
Note that sometimes we will  use $E+i$ as well for the same concept. While using a negative shift, the reader should be careful to distinguish between $E-i$ and the set difference $E-\{i\}=E\setminus \{i\}$.

A {\em cyclic shift} of a composition $\alpha=(\alpha_1,\alpha_2,\ldots,\alpha_m)$ is a composition of the form
\[
(\alpha_k,\ldots,\alpha_m,\alpha_1,\ldots,\alpha_{k-1})
\]
for some $k\in[m]$.
We adopt the notations from \cite{agrr:cqf} and denote by $c2^{[n]}_0$ (respectively, $\cComp_n$)  the set of equivalence classes of elements of $2^{[n]}_0$ (respectively, $\Comp_n$) under cyclic shifts. 
Here we recall two natural bijections which will play important roles when indexing two particular bases of (cyclic) quasi-symmetric functions.

The first natural bijection is between $2^{[n-1]}$ and $\Comp_n$. The map $\Phi:2^{[n-1]} \to\Comp_n$ is defined by 
\begin{equation}\label{eq:Phi} 
  \Phi(E):=(e_1-e_0,e_2-e_1,\ldots,e_k-e_{k-1},e_{k+1}-e_k)  
\end{equation}
for any given $E=\{e_1<e_2<\dots<e_k\} \sbe[n-1]$ with $e_0=0$ and $e_{k+1}=n$, where the inverse map is
\[
   \Phi^{-1}(\alpha)=\{\alpha_1,\alpha_1+\alpha_2,\ldots,\alpha_1+\alpha_2+\dots+\alpha_k\} 
\] 
for any $\alpha=(\alpha_1,\ldots,\alpha_{k+1})\vDash n$.

Another bijection is between $c2_0^{[n]}$ and $\cComp_n$, for the sake of which we need to consider the map $\psi:2_0^{[n]}\to\Comp_n$ defined by
\begin{equation}\label{eq:psi}
\psi(E):=(e_2-e_1,\ldots,e_k-e_{k-1},e_1-e_k+n)
\end{equation}
where $E=\{e_1<e_2<\dots<e_k\} \sbe[n]$. Notice that if $E'$ is a cyclic shift of $E$ in $[n]$, then $\psi(E')$ is also a cyclic shift of $\psi(E)$.
Therefore $\psi$ induces a map $\Psi:c2_0^{[n]}\to\cComp_n$. Moreover, it is straightforward to check that the induced map $\Psi$ is bijective.

\subsection{Quasi-symmetric functions $\QSym$}

A {\em quasi-symmetric function} is a formal power series $f\in \mathbb{Q}[[x_1,x_2,\ldots]]$ 
such that for any sequence of positive integers $a=(a_1,a_2,\ldots,a_s)$, and two increasing sequences $i_1<i_2<\dots< i_s$ and $j_1<j_2<\dots< j_s$ of positive integers,
\[
[x_{i_1}^{a_1}x_{i_2}^{a_2}\dots x_{i_s}^{a_s}]\,f=[x_{j_1}^{a_1}x_{j_2}^{a_2}\dots x_{j_s}^{a_s}]\,f,
\]
where $[x_{i_1}^{a_1}x_{i_2}^{a_2}\dots x_{i_s}^{a_s}]\,f$ denotes the coefficient of monomial $x_{i_1}^{a_1}x_{i_2}^{a_2}\dots x_{i_s}^{a_s}$ in the expression of $f$. Let $\QSym_n$ be the set of all quasi-symmetric functions which are homogeneous of degree $n$, and $\QSym = \oplus_{n\ge0} \QSym_n$.
Two bases of $\QSym$ are particularly important to our work: monomial quasi-symmetric functions $M_L$ and fundamental quasi-symmetric functions $F_L$.

Given a composition $\alpha=(\alpha_1,\alpha_2,\ldots,\alpha_s)\vDash n$, the associated {\em monomial quasi-symmetric function} indexed by $\alpha$ is
\[
    M_\alpha=\sum_{i_1<i_2<\dots<i_s} x_{i_1}^{\alpha_1}x_{i_2}^{\alpha_2}\dots x_{i_s}^{\alpha_s}. 
\]
It is clear that $\{M_\alpha\}_{\alpha\vDash n}$ form a basis of $\QSym_n$. From the natural bijection $\Phi:2^{[n-1]} \to\Comp_n$ defined by~\eqref{eq:Phi}, we can also index the monomial quasi-symmetric functions by subsets $E\sbe[n-1]$, and define $M_{n,E}:=M_{\Phi(E)}$.

There is another important basis of $\QSym$.  The {\em fundamental quasi-symmetric function} indexed by $E\sbe[n-1]$ is 
\[
    F_{n,E}=\sum_{\substack{i_1\leq\dots\leq i_n\\[4pt] i_k<i_{k+1}\, \text{if $k\in E$}}} x_{i_1}x_{i_2}\dots x_{i_n}. 
\]
Similarly, we can define $F_\alpha:=F_{\Phi^{-1}(\alpha)}$ indexed by compositions.

The relation between monomial and fundamental quasi-symmetric functions is simple: 
\begin{equation}\label{eq:M vs F}
    F_{n,E}=\sum_{L\supseteq E}M_{n,L}.
\end{equation}
By the principle of inclusion and exclusion, 
$M_{n,E}$ can be expressed as a linear combination of the $F_{n,L}$, from which we can tell that $\{F_{n,L}\}_{L\sbe[n-1]}$ spans $\QSym_n$. By checking the cardinality of both sets $\{F_{n,L}\}_{L\sbe[n-1]}$ and $\{M_{n,E}\}_{E\sbe[n-1]}$, it follows that $\{F_{n,L}\}_{L\sbe[n-1]}$ is indeed a basis of $\QSym_n$.

\begin{example}
Consider $E=\{1,3\}$, by definition we have 
\[
F_{4,\{1,3\}}=\sum_{i_1<i_2\leq i_3<i_4} x_{i_1}x_{i_2}x_{i_3}x_{i_4}=\sum\limits_{i_1<i_2< i_3<i_4} x_{i_1}x_{i_2}x_{i_3}x_{i_4}+\sum_{i_1<i_2<i_4} x_{i_1}x_{i_2}^2x_{i_4}.
\]
There are only two choices for a set $L$ satisfying that $E\sbe L\sbe[3]$: $\{1,3\}$ or $\{1,2,3\}$. Since $\Phi(\{1,3\})=(1,2,1),\Phi(\{1,2,3\})=(1,1,1,1)\,$, we get
\[
M_{4,\{1,3\}}=M_{(1,2,1)}=\sum_{i_1<i_2<i_3} x_{i_1}x_{i_2}^2x_{i_3},\hs{5pt} M_{4,\{1,2,3\}}=M_{(1,1,1,1)}=\sum_{i_1<i_2< i_3<i_4} x_{i_1}x_{i_2}x_{i_3}x_{i_4}.
\]
The calculation above verifies that $F_{4,\{1,3\}}=M_{4,\{1,3\}}+M_{4,\{1,2,3\}}=\sum\limits_{L\supseteq \{1,3\}}M_{4,L}$.
\end{example}

\subsection{Cyclic quasi-symmetric functions $\cQSym$}

In this subsection, we recall from~\cite{agrr:cqf} the theory of cyclic quasi-symmetric functions. We will model our work of enriched toric $[\vec{D}]$-partitions with enumerators in this environment. 

A {\em cyclic quasi-symmetric function} is a formal power series $f\in \mathbb{Q}[[x_1,x_2,\ldots]]$ 
such that for any sequence of positive integers $a=(a_1,a_2,\ldots,a_s)$, a cyclic shift $(a'_1,a'_2,\ldots,a'_s)$ of $a$, and two increasing sequences $i_1<i_2<\dots <i_s$ and $j_1<j_2<\dots <j_s$ of positive integers, 

\[
[x_{i_1}^{a_1}x_{i_2}^{a_2}\dots x_{i_s}^{a_s}]\,f=[x_{j_1}^{a'_1}x_{j_2}^{a'_2}\dots x_{j_s}^{a'_s}]\,f,
\]
namely the coefficients of $x_{i_1}^{a_1}x_{i_2}^{a_2}\dots x_{i_s}^{a_s}$ and $x_{j_1}^{a'_1}x_{j_2}^{a'_2}\dots x_{j_s}^{a'_s}$ in $f$ are equal. Denote by $\cQSym_n$ the set of all cyclic quasi-symmetric functions which are homogeneous of degree $n$, and $\cQSym = \oplus_{n\ge0} \cQSym_n$.

\begin{remark}
It is clear that there exists a strict inclusion relation $\sym\subsetneq\cQSym\subsetneq\QSym$, where $\sym$ is the algebra of symmetric functions.
\end{remark}

We have the following cyclic analogues of the concepts of monomial (fundamental) quasi-symmetric functions.

Given a composition $\alpha=(\alpha_1,\alpha_2,\ldots,\alpha_s)\vDash n$, the associated {\em monomial cyclic quasi-symmetric function} indexed by $\alpha$ is
\[
    M^{\cyc}_\alpha=\sum_{i=1}^{s} M_{(\alpha_i,\alpha_{i+1},\ldots,\alpha_{i-1})},
\]
where the indices are interpreted modulo 
$s$, meaning $\alpha_{j}=\alpha_{j+s}$. In other words, $M^{\cyc}_\alpha$ sums over all monomial quasi-symmetric functions indexed by cyclic shifts of $\alpha$. Therefore it is clear that $M^{\cyc}_\alpha=M^{\cyc}_{\alpha'}$ if $\alpha$ and $\alpha'$ only differ by a cyclic shift.

We can also index the monomial cyclic quasi-symmetric function by sets. For a nonempty $E\sbe[n]$, define $M^{\cyc}_{n,E}:=M^{\cyc}_{\psi(E)}$ via the map $\psi:2_0^{[n]}\to\Comp_n$ defined by~\eqref{eq:psi}, and set $M^{\cyc}_{n,\emptyset}:=0$. Similarly it can be shown that $M^{\cyc}_{n,E}=M^{\cyc}_{n,E'}$ if $E'$ is a cyclic shift of $E$.

The following result gives the expression of monomial cyclic quasi-symmetric functions in terms of monimial quasi-symmetric functions.

\begin{lem}[\cite{agrr:cqf} Lemma 2.5, monomial to cyclic monomial]
For any subset $E\sbe[n]$
\begin{equation}\label{eq:monomial to cyclic monomial}
M^{\cyc}_{n,E}=\sum_{e\in E} M_{n,(E-e)\cap[n-1]},
\end{equation}
where the set $E-e$ is defined as~\eqref{eq:i+E}.
\qed
\end{lem}

\begin{example}
Table~\ref{table: cyclic M} computes all monomial cyclic quasi-symmetric function indexed by compositions of $4$, in terms of monomial quasi-symmetric functions.
\begin{table}
\caption{Monomial cyclic quasi-symmetric functions indexed by compositions of $4$}
    \label{table: cyclic M}
    $$
    \begin{array}{|c|c|}
    \hline
        \alpha\vDash 4 &  M^{\cyc}_\alpha\\
        \hline \hline
        (4) & M_{(4)}=M_{4,\emptyset}\\
        \hline
        \text{(1,3) or (3,1)}  & M_{(1,3)}+M_{(3,1)}=M_{4,\{1\}}+M_{4,\{3\}}\\
        \hline
        (2,2)&
        2M_{(2,2)}=2M_{4,\{2\}}\\
        \hline
        \text{(1,1,2) or (1,2,1) or (2,1,1)} & M_{(1,1,2)}+M_{(1,2,1)}+M_{(2,1,1)}=M_{4,\{1,2\}}+M_{4,\{1,3\}}+M_{4,\{2,3\}}\\
        \hline
        (1,1,1,1) & 4M_{(1,1,1,1)}=4M_{4,\{1,2,3\}}\\
        \hline
    \end{array}
    $$
\end{table}
\end{example}

For the natural desire of establishing a similar relation as the one between monomial and fundamental quasi-symmetric functions given by~\eqref{eq:M vs F} in our cyclic situation, define the {\em fundamental cyclic quasi-symmetric function} indexed by $E\sbe[n]$ as

\begin{equation}\label{relation between F and M}
F_{n,E}^{\cyc}:=\sum_{L\supseteq E}M_{n,L}^{\cyc}.    
\end{equation}

\begin{remark}
This is not the original definition in~\cite{agrr:cqf} but appears as a lemma in the same article. But these two definitions are equivalent via~\cite{agrr:cqf} Lemma 2.14. We will use this definition for the purpose of our work, and mention the original definition in Proposition~\ref{orginal def of cyclic} for interested readers.
\end{remark}

\begin{example}\label{ex:cyclic F}
Consider $n=4$ and $E=\{1,3\}$. By definition
\begin{align*}
F_{4,\{1,3\}}^{\cyc}=&\quad\sum_{L\supseteq \{1,3\}}M_{4,L}^{\cyc}\\
\overset{(i)}{=}& \quad M_{4,\{1,3\}}^{\cyc}+M_{4,\{1,2,3\}}^{\cyc}+M_{4,\{1,3,4\}}^{\cyc}+M_{4,\{1,2,3,4\}}^{\cyc}\\
\overset{(ii)}{=}&\quad M_{(2,2)}^{\cyc}+M_{(1,1,2)}^{\cyc}+M_{(2,1,1)}^{\cyc}+M_{(1,1,1,1)}^{\cyc}  \\
\overset{(iii)}{=}&\quad 2M_{(2,2)}+2\left(M_{(1,1,2)}+M_{(1,2,1)}+M_{(2,1,1)})\right)+4M_{(1,1,1,1)}\\
=&\quad 2\sum_{i_1<i_2}x_{i_1}^2x_{i_2}^2 \\
&\quad+2\sum_{i_1<i_2< i_3}(x_{i_1}x_{i_2}x_{i_3}^2+x_{i_1}x_{i_2}^2x_{i_3}+x_{i_1}^2x_{i_2}x_{i_3})\\
&\quad+4\sum_{i_1<i_2< i_3<i_4}x_{i_1}x_{i_2}x_{i_3} x_{i_4}.
\end{align*}
Equality $(i)$ follows from the fact that the choices for $L\supseteq \{1,3\}$ in $[4]$ are $\{1,3\}$, $\{1,2,3\}$, $\{1,3,4\}$ and $\{1,2,3,4\}$. Equality $(ii)$ is obtained by changing indices under the map $\psi$ defined by~\eqref{eq:psi}, equality (iii) is from Table~\ref{table: cyclic M}.
\end{example}


The following passing from fundamental to cyclic fundamental quasi-symmetric functions should come without surprise.

\begin{lem}[\cite{agrr:cqf} Proposition 2.15, fundamental to cyclic fundamental]\label{lem: F to cylic F}
For any subset $E\sbe[n]$,
\[
F^{\cyc}_{n,E}=\sum_{i\in [n]} F_{n,(E-i)\cap[n-1]},
\]
with set $E-i$ defined by~\eqref{eq:i+E}.
\qed
\end{lem}

\begin{remark}
\hfill
\begin{enumerate}
    \item It follows directly from Lemma~\ref{lem: F to cylic F} that $F^{\cyc}_{n,E}=F^{\cyc}_{n,E'}$ if $E'$ is a cyclic shift of $E$.
    \item Clearly the set $\{M^{\cyc}_{n,E}: E\in c2^{[n]}_0\}$ spans $\cQSym_n$ and is linearly independent, as each monomial of degree $n$ appears in $M^{\cyc}_{n,E}$ for exactly one $E\in c2^{[n]}_0$. Hence $\{M^{\cyc}_{n,E}: E\in c2^{[n]}_0\}$ is a basis of $\cQSym_n$. Applying the principle of inclusion and exclusion on~\eqref{relation between F and M} we have
    $$M^{\cyc}_{n,E}=\sum_{L\supseteq E}(-1)^{|L\setminus E|}F^{\cyc}_{n,L}, $$
    which implies that $\{F^{\cyc}_{n,E}: E\in c2^{[n]}_0\}$ also spans $\cQSym_n$, 
    together with the fact that the dimension of vector space $\cQSym_n$ is $\#c2^{[n]}_0$, $\{F^{\cyc}_{n,E}: E\in c2^{[n]}_0\}$ is also a basis of $\cQSym_n$.
\end{enumerate}

\end{remark}


Now we review the original definition of cyclic quasi-symmetric functions. Before that, we need to associate each $E\sbe[n]$ with a set $P_{n,E}^{\cyc}$.

\begin{definition}
Let $P_{n,E}^{\cyc}$ denote the set of all pairs $(w,k)$ where $w=w_1\ldots w_n$ a sequence of positive integers and  $k\in[n]$ satisfying 
\begin{enumerate}
    \item[(1)] $w$ is ``cyclically weakly increasing" from index $k$, 
    that is, $w_k\leq w_{k+1}\leq\ldots\leq w_n\leq w_1\leq\ldots\leq w_{k-1}$.
    \item[(2)] If $i\in E\setminus\{k-1\}$, then $w_i<w_{i+1}$, where the indices are computed modulo $n$. 
\end{enumerate}
\end{definition}

\begin{remark}
The index $k$ is uniquely defined by $w$ unless all integers in $w$ are the same. In that case, it must be either $E=\{k-1\}$ for some $k\in[n]$ or $E=\emptyset$. Those $w$ with elements all equal will only be paired with $k-1$ if $E=\{k-1\}$, and so get counted once in $P_{n,E}^{\cyc}$; as for $E=\emptyset$, if $w$ has all elements equal, it can pair with every $k\in[n]$, 
therefore it will be counted $n$ times in $P_{n,E}^{\cyc}$.
\end{remark}

\begin{prop}\label{orginal def of cyclic}
For every subset $E\sbe[n]$,
\[
F_{n,E}^{\cyc}=\sum_{(w,k)\in P_{n,E}^{\cyc}}x_{w_1}x_{w_2}\dots x_{w_n}.
\]
\end{prop}

\begin{proof}
Despite the statement, the proof is essentially the same as for~\cite{agrr:cqf} Lemma 2.14.
\end{proof}

Here we provide a concrete example to illustrate the validity of the proposition above.
\begin{example}
Consider $n=4$ and $E=\{1,3\}$. It is clear that $P_{n,E}^{\cyc}$ has a natural partition into four parts where each part contains all pairs with a given index $k\in[4]$. This implies the following summation
\begin{align*}
    \sum_{(w,k)\in P_{4,\{1,3\}}^{\cyc}}x_{w_1}x_{w_2}x_{w_3}x_{w_4}&=\sum_{k=1}^4\sum_{(w,k)\in P_{4,\{1,3\}}^{\cyc}}x_{w_1}x_{w_2}x_{w_3}x_{w_4}\\ 
    &=\sum_{w_1<w_2\leq w_3<w_4}x_{w_1}x_{w_2}x_{w_3} x_{w_4}+\sum_{w_2\leq w_3<w_4\leq w_1}x_{w_1}x_{w_2}x_{w_3} x_{w_4}\\
    &\qquad +\sum_{w_3<w_4\leq w_1<w_2}x_{w_1}x_{w_2}x_{w_3} x_{w_4}+\sum_{w_4\leq w_1<w_2\leq w_3}x_{w_1}x_{w_2}x_{w_3} x_{w_4},
\end{align*}
For each summand,
\begin{align*}
\begin{split}
\sum_{w_1<w_2\leq w_3<w_4}x_{w_1}x_{w_2}x_{w_3} x_{w_4}&=\sum_{w_1<w_2< w_3<w_4}x_{w_1}x_{w_2}x_{w_3} x_{w_4}+\sum_{w_1<w_2= w_3<w_4}x_{w_1}x_{w_2}x_{w_3} x_{w_4}\\
&=\sum_{w_1<w_2< w_3<w_4}x_{w_1}x_{w_2}x_{w_3} x_{w_4}+\sum_{w_1<w_2<w_3}x_{w_1}x_{w_2}^2 x_{w_3};
\end{split}
\\[2ex]
\begin{split}
\sum_{w_2\leq w_3<w_4\leq w_1}x_{w_1}x_{w_2}x_{w_3} x_{w_4}&=\sum_{w_2<w_3< w_4<w_1}x_{w_2}x_{w_3}x_{w_4} x_{w_1}+\sum_{w_2= w_3<w_4< w_1}x_{w_2}x_{w_3}x_{w_4}x_{w_1}\\
&\qquad+\sum_{w_2< w_3<w_4= w_1}x_{w_2}x_{w_3}x_{w_4}x_{w_1}+\sum_{w_2= w_3<w_4= w_1}x_{w_2}x_{w_3}x_{w_4}x_{w_1}\\
&=\sum_{i_1<i_2< i_3<i_4}x_{i_1}x_{i_2}x_{i_3} x_{i_4}+\sum_{i_1<i_2< i_3}x_{i_1}^2x_{i_2}x_{i_3}+\sum_{i_1<i_2< i_3}x_{i_1}x_{i_2}x_{i_3}^2+\sum_{i_1<i_2}x_{i_1}^2x_{i_2}^2;
\end{split}
\\[2ex]
\begin{split}
    \sum_{w_3<w_4\leq w_1<w_2}x_{w_1}x_{w_2}x_{w_3} x_{w_4}&=\sum_{w_3<w_4< w_1<w_2}x_{w_3} x_{w_4}x_{w_1}x_{w_2}+\sum_{w_3<w_4= w_1<w_2}x_{w_3} x_{w_4}x_{w_1}x_{w_2}\\
&=\sum_{i_1<i_2< i_3<i_4}x_{i_1}x_{i_2}x_{i_3} x_{i_4}+\sum_{w_1<w_2<w_3}x_{w_1}x_{w_2}^2 x_{w_3};
\end{split}
\\[2ex]
\begin{split}
\sum_{w_4\leq w_1<w_2\leq w_3}x_{w_1}x_{w_2}x_{w_3} x_{w_4}&=\sum_{w_4<w_1< w_2<w_3}x_{w_4}x_{w_1}x_{w_2} x_{w_3}+\sum_{w_4= w_1<w_2< w_3}x_{w_4}x_{w_1}x_{w_2} x_{w_3}\\
&\qquad+\sum_{w_4< w_1<w_2= w_3}x_{w_4}x_{w_1}x_{w_2} x_{w_3}+\sum_{w_4= w_1<w_2= w_3}x_{w_4}x_{w_1}x_{w_2} x_{w_3}\\
&=\sum_{i_1<i_2< i_3<i_4}x_{i_1}x_{i_2}x_{i_3} x_{i_4}+\sum_{i_1<i_2< i_3}x_{i_1}^2x_{i_2}x_{i_3}+\sum_{i_1<i_2< i_3}x_{i_1}x_{i_2}x_{i_3}^2+\sum_{i_1<i_2}x_{i_1}^2x_{i_2}^2.
\end{split}
\end{align*}
To sum up, we have
\begin{equation}\label{ex: F1,3}
\begin{aligned}
    \sum_{(w,k)\in P_{4,\{1,3\}}^{\cyc}}x_{w_1}x_{w_2}x_{w_3}x_{w_4}&= 4\sum_{i_1<i_2< i_3<i_4}x_{i_1}x_{i_2}x_{i_3} x_{i_4} \\
    &\qquad+2\sum_{i_1<i_2< i_3}(x_{i_1}^2x_{i_2}x_{i_3}+x_{i_1}x_{i_2}^2x_{i_3}+x_{i_1}x_{i_2}x_{i_3}^2)\\
    &\qquad+2\sum_{i_1<i_2}x_{i_1}^2x_{i_2}^2,
\end{aligned}
\end{equation}
which is exactly the expression of $F_{4,\{1,3\}}^{\cyc}$ we obtained in Example~\ref{ex:cyclic F}.
\end{example}

\section{Enriched $\vec{D}$-partitions for DAGs}\label{ep}
In \cite{S:epp}, Stembridge defined the enriched $P$-partition of a poset $P$. Different from having $\mathbb{P}$ with the ordinary order $\leq$ 
as range for ordinary $P$-partitions, enriched $P$-partitions are obtained by imposing another total order on the range, $\mathbb{P'}$, which is the set of nonzero integers in an unusual order. We review the basic theory from Stembridge and naturally extend enriched $P$-partitions to enriched $\vec{D}$-partitions where $\vec{D}$ is a directed acyclic graph which is not necessarily transitive.

A {\em directed acyclic graph} (DAG) is a digraph with no directed cycles. Suppose $\vec{D}$ is a DAG with vertex set $[n]$. A DAG $\vec{D}$ is {\em transitive} if $i\to j$ and $j\to k$ implies $i\to k$ for $i,j,k\in[n]$. 
If $\vec{P}$ is a transitive DAG, 
it will naturally induce a partial order $<_{\vec{P}}$ on the vertex set $[n]$, defined so that for two vertices $i$ and $j$, one has $i<_{\vec{P}}j$ if and only if  $i\to j$ in $\vec{P}$; in that case, we also use $\vec{P}$ to denote this poset. 

In general, we can associate a partial order with any DAG $\vec{D}$. For this purpose, define the {\em transitive closure} $\vec{P}$ of $\vec{D}$ as the directed graph obtained from $\vec{D}$ by adding in $i\to k$ if one has both $i\to j$ and $j\to k$ in $\vec{D}$. Such $\vec{P}$ is unique. Moreover, it is straightforward to verify that the transitive closure $\vec{P}$ is both acyclic and transitive, where the acyclicity is ensured by the acyclic condition of $\vec{D}$ and the transitivity comes from the added conditions. This implies that $\vec{P}$ is actually a transitive DAG, hence has a partial order structure. We will maintain the use $\to$ for the relation between vertices in a general DAG $\vec{D}$, while using the partial order $\leq_{\vec{P}}$ 
on a poset $\vec{P}$.

A poset $\vec{P}$ is a {\em total linear order} if it is a complete DAG, i.e., there is a directed edge between every pair of vertices in $\vec{P}$. There is a bijection between the set of total linear orders $\vec{P}$ on the vertex set $[n]$ and $\mathcal{S}_n$. 
For a total linear order $\vec{P}$ on $[n]$, there exists a unique directed path $w_1\to w_2\to\dots\to w_n$ in $\vec{P}$, hence $\vec{P}$ can be identified with the permutation $w=w_1 w_2\dots w_n\in \mathcal{S}_n$. Conversely, given a permutation $w=w_1 w_2\dots w_n\in\mathcal{S}_n$, we can construct a DAG $\vec{P}$ by putting arrows $w_i\to w_j$ for all $1\leq i<j\leq n$ on the vertex set $[n]$, and it is easy to check that the resulting DAG $\vec{P}$ is actually a total linear order. In this case, we usually use a permutation $w$ to denote the corresponding total linear order $\vec{P}$.

For two DAGs $\vec{D}_1$ and $\vec{D}_2$ on the same vertex set $[n]$, we say $\vec{D}_2$ {\em extends} $\vec{D}_1$ if $\vec{D}_1$ is a subgraph of $\vec{D}_2$, written as $\vec{D}_1\sbe\vec{D}_2$. If, furthermore, $\vec{D}_2$ is a total linear order corresponding to the permutation $w\in\mathcal{S}_n$, we also say that $w$ {\em linearly extends} $\vec{D}_1$. Denote by $\mathcal{L}(\vec{D})$ the set of all permutations $w\in\mathcal{S}_n$ which linearly extend $\vec{D}$.

\begin{example} \label{DAG on [4]}
Here are several DAGs on the vertex set $[4]=\{1,2,3,4\}$.

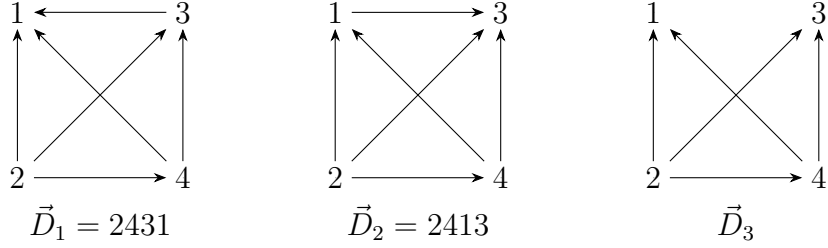
\begin{center}
        \begin{tikzpicture}[scale=1.1, yshift=10in]
        \draw[-{Stealth[scale=1.0]}] (0.2,0.2) -- (1.8,1.8);
        \draw[-{Stealth[scale=1.0]}] (0,0.2) -- (0,1.8);
        \draw[-{Stealth[scale=1.0]}] (0.2,0) -- (1.8,0) node[pos=.5, below=6pt] {$\vec{D}_1=2431$};
        \draw[-{Stealth[scale=1.0]}] (2,0.2) -- (2,1.8);
        \draw[-{Stealth[scale=1.0]}] (1.8,0.2) -- (0.2,1.8);
        \draw[-{Stealth[scale=1.0]}] (1.8,2) -- (0.2,2);
        \foreach \Point/\PointLabel in {(0,0)/2, (2,0)/4, (2,2)/3, (0,2)/1}
            \draw[fill=black] \Point circle (0) node {$\PointLabel$};
        \end{tikzpicture}
        \hspace{0.5in}
        \begin{tikzpicture}[scale=1.1]
         \draw[-{Stealth[scale=1.0]}] (0.2,0.2) -- (1.8,1.8);
        \draw[-{Stealth[scale=1.0]}] (0,0.2) -- (0,1.8);
        \draw[-{Stealth[scale=1.0]}] (0.2,0) -- (1.8,0) node[pos=.5, below=6pt] {$\vec{D}_2=2413$};
        \draw[-{Stealth[scale=1.0]}] (2,0.2) -- (2,1.8);
        \draw[-{Stealth[scale=1.0]}] (1.8,0.2) -- (0.2,1.8);
        \draw[-{Stealth[scale=1.0]}] (0.2,2)--(1.8,2) ;
        \foreach \Point/\PointLabel in {(0,0)/2, (2,0)/4, (2,2)/3, (0,2)/1}
            \draw[fill=black] \Point circle (0) node {$\PointLabel$};
         \end{tikzpicture}
         \hspace{0.5in}
        \begin{tikzpicture}[scale=1.1]
         \draw[-{Stealth[scale=1.0]}] (0.2,0.2) -- (1.8,1.8);
        \draw[-{Stealth[scale=1.0]}] (0,0.2) -- (0,1.8);
        \draw[-{Stealth[scale=1.0]}] (0.2,0) -- (1.8,0) node[pos=.5, below=6pt] {$\vec{D}_3$};
        \draw[-{Stealth[scale=1.0]}] (2,0.2) -- (2,1.8);
        \draw[-{Stealth[scale=1.0]}] (1.8,0.2) -- (0.2,1.8);
        \foreach \Point/\PointLabel in {(0,0)/2, (2,0)/4, (2,2)/3, (0,2)/1}
            \draw[fill=black] \Point circle (0) node {$\PointLabel$};
         \end{tikzpicture}
\end{center}
Note that $\vec{D}_1$ is the total linear order with permutation $2431$, $\vec{D}_2$ is the total linear order for $2413$. Both $\vec{D}_1$ and $\vec{D}_2$ extend $\vec{D}_3$, hence $2431$ and $2413$ linearly extend $\vec{D}_3$. Moreover, $\vec{D}_1$ and $\vec{D}_2$ are the only total linear orders which extend $\vec{D}_3$, 
therefore $\mathcal{L}(\vec{D}_3)=\{2431,2413\}$.
\end{example}

Now we are in a good position to define  enriched $\vec{D}$-partitions for a DAG $\vec{D}$. Stembridge originally defined the enriched $P$-partitions when $P$ is a poset. This definition can be easily extended to the cases when $\vec{D}$ is simply a DAG, as the definition does not rely on the transitivity of $P$.

Stembridge defines $\mathbb{P}'$ to be the set of nonzero integers, totally ordered as
$$-1\prec 1\prec -2\prec 2\prec -3\prec 3\prec\cdots.$$

\begin{definition}[Enriched $\vec{D}$-partition] \label{def:enriched DP}
Let $\vec{D}$ be a directed acyclic graph (DAG) on $[n]$. An {\em enriched $\vec{D}$-partition} is a function $f:[n]\to \mathbb{P'}$ such that for all $i\to j$ in $\vec{D}$,
\begin{enumerate}
    \item[(a)] $f(i)\preceq f(j)$, 
    \item[(b)] $f(i)=f(j)>0$ implies $i<j$,
    \item[(c)] $f(i)=f(j)<0$ implies $i>j$.
\end{enumerate}
Denote by $\mathcal{E}(\vec{D})$ the set of all enriched $\vec{D}$-partitions $f$.
\end{definition}

\begin{remark}
\hfill
\begin{enumerate}
    \item In this definition we are using two order structures on the domain $[n]$: the order $\to$ induced by DAG ${\vec{D}}$ and the ordinary total order $\leq$ on integers in (b) and (c). Both of them will impose restrictions on the possible choices for $f$. As for the range $\mathbb{P'}$, we also use two order structures: the total order $\preceq$ defined by Stembridge in (a) and the usual order $\leq$ on the integers in (b) and (c). 
    \item If $\vec{D}=w$ is a total linear order, the structure of the set of enriched $w$-partitions is quite simple:
\begin{equation}\label{eq:enriched}
\begin{array}{rl} 
    \mathcal{E}(w)=\{\,f:[n]\to \mathbb{P'} ~| & f(w_1)\preceq\cdots\preceq f(w_n),  \\
     & f(w_i)=f(w_{i+1})>0 \Rightarrow i\notin \Des(w),\\
     & f(w_i)=f(w_{i+1})<0 \Rightarrow i\in \Des(w)\,\}.\\
\end{array} 
\end{equation}
\end{enumerate}
\end{remark}

The following fundamental lemma is a straightforward analogue of Stembridge \cite{S:epp}, Lemma 2.1.

\begin{lem} [Fundamental lemma of enriched $\vec{D}$-partitions]
\label{FLEDP}
For any DAG $\vec{D}$ with vertex set $[n]$, one has a decomposition of $\mathcal{E}(\vec{D})$ as the following disjoint union:
$$\mathcal{E}(\vec{D})=\bigsqcup_{w\in \mathcal{L}(\vec{D})}\mathcal{E}(w).$$
\end{lem}
\begin{proof}
Given an enriched $\vec{D}$-partition $f$. First we arrange the elements of $[n]$ in a weakly increasing order of $f$-values with respect to the total order $\preceq$ on the range. Then if some elements in $[n]$ have the same $f$-value $-k$ (respectively, $+k$) for some positive integer $k$, we arrange them in a decreasing (respectively, increasing) order with respect to the usual order $\leq$ on the domain. The resulting permutation $w$ is unique with $f\in\mathcal{E}(w)$, and $w$ linearly extends $\vec{D}$. On the other hand, for $w\in \mathcal{L}(\vec{D})$, every enriched $w$-partition is also an enriched $\vec{D}$-partition. Therefore the conclusion follows.
\end{proof}

\begin{example}
Returning to Example \ref{DAG on [4]}, since $\mathcal{L}(\vec{D}_3)=\{2431,2413\}$, 
by the Fundamental Lemma,
$\mathcal{E}(\vec{D}_3)=\mathcal{E}(2431)\uplus \mathcal{E}(2413)$.
\end{example}

\section{Enriched toric $[\vec{D}]$-partitions for toric DAGs}\label{etp}
In this section, we review the toric DAGs and toric posets as cyclic analogues of DAGs and posets. Then we define enriched toric $[\vec{D}]$-partitions and develop some of their properties. The concept of toric poset was originally defined and studied by Develin, Macauley and Reiner in \cite{Dmr:tpo}. Here we follow the presentation from Adin, Gessel, Reiner and Roichman \cite{agrr:cqf}.

Just like a linear permutation has a corresponding cyclic permutation as the equivalence class under the equivalence of rotation, 
for a DAG, we will define an equivalence relation and consider the equivalence class to be the corresponding toric DAG. It turns out that if $w$ is a linear extension of the DAG $\vec{D}$, then $[w]$ is a toric extension of the corresponding toric DAG $[\vec{D}]$.

A DAG $\vec{D}$ on $[n]$ has $i_0\in[n]$ as a {\em source} (respectively, {\em sink}) if $\vec{D}$ does not contain $j\to i_0$ (respectively, $i_0\to j$) for any $j\in[n]$. Suppose $i_0$ is a source or a sink in $\vec{D}$, we say $\vec{D}'$ is obtained from $\vec{D}$ by a {\em flip} at $i_0$ if $\vec{D}'$ is obtained by reversing all arrows containing $i_0$. We define the equivalence relation $\equiv$ on DAGs as follows:
    $\vec{D}'\equiv\vec{D}$ if and only if $\vec{D}'$ is obtained from $\vec{D}$ by a sequence of flips. 
A {\em toric} DAG is the equivalence class $[\vec{D}]$ of a DAG $\vec{D}$.

In particular, if $\vec{D}=w=w_1 w_2\ldots w_n$ is a total linear order, the next proposition claims that the corresponding toric DAG $[\vec{D}]$ can be identified with the cyclic permutation $[w]$.

\begin{prop}[\cite{Dmr:tpo}, Proposition 4.2]\label{prop:cp}
If $\vec{D}=w$ is a total linear order with $w=w_1\ldots w_n$, then there is a bijection between toric DAG $[\vec{D}]$ and cyclic permutation $[w]$.\qed
\end{prop}

\begin{example}
The total linear order $\vec{D}_1=2431$ from Example \ref{DAG on [4]} has a corresponding toric DAG $[\vec{D}_1]$:

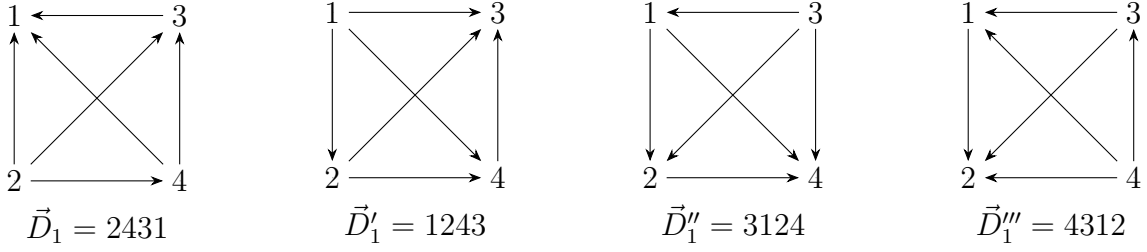
\begin{center}
        \begin{tikzpicture}[scale=1.1, yshift=10in]
        \draw[-{Stealth[scale=1.0]}] (0.2,0.2) -- (1.8,1.8);
        \draw[-{Stealth[scale=1.0]}] (0,0.2) -- (0,1.8);
        \draw[-{Stealth[scale=1.0]}] (0.2,0) -- (1.8,0) node[pos=.5, below=6pt] {$\vec{D}_1=2431$};
        \draw[-{Stealth[scale=1.0]}] (2,0.2) -- (2,1.8);
        \draw[-{Stealth[scale=1.0]}] (1.8,0.2) -- (0.2,1.8);
        \draw[-{Stealth[scale=1.0]}] (1.8,2) -- (0.2,2);
        \foreach \Point/\PointLabel in {(0,0)/2, (2,0)/4, (2,2)/3, (0,2)/1}
            \draw[fill=black] \Point circle (0) node {$\PointLabel$};
        \end{tikzpicture}
        \hspace{0.5in}
        \begin{tikzpicture}[scale=1.1]
        \draw[-{Stealth[scale=1.0]}] (0.2,0.2) -- (1.8,1.8);
        \draw[-{Stealth[scale=1.0]}] (0,1.8) -- (0,0.2);
        \draw[-{Stealth[scale=1.0]}] (0.2,0) -- (1.8,0) node[pos=.5, below=6pt] {$\vec{D}_1'=1243$};
        \draw[-{Stealth[scale=1.0]}] (2,0.2) -- (2,1.8);
        \draw[-{Stealth[scale=1.0]}] (0.2,1.8) -- (1.8,0.2);
        \draw[-{Stealth[scale=1.0]}] (0.2,2)--(1.8,2) ;
        \foreach \Point/\PointLabel in {(0,0)/2, (2,0)/4, (2,2)/3, (0,2)/1}
            \draw[fill=black] \Point circle (0) node {$\PointLabel$};
         \end{tikzpicture}
         \hspace{0.5in}
        \begin{tikzpicture}[scale=1.1]
         \draw[-{Stealth[scale=1.0]}](1.8,1.8) -- (0.2,0.2);
        \draw[-{Stealth[scale=1.0]}] (0,1.8) -- (0,0.2);
        \draw[-{Stealth[scale=1.0]}] (0.2,0) -- (1.8,0) node[pos=.5, below=6pt] {$\vec{D}_1''=3124$};
        \draw[-{Stealth[scale=1.0]}] (2,1.8) -- (2,0.2);
        \draw[-{Stealth[scale=1.0]}] (1.8,2) -- (0.2,2);
        \draw[-{Stealth[scale=1.0]}] (0.2,1.8) -- (1.8,0.2);
        \foreach \Point/\PointLabel in {(0,0)/2, (2,0)/4, (2,2)/3, (0,2)/1}
            \draw[fill=black] \Point circle (0) node {$\PointLabel$};
         \end{tikzpicture}
         \hspace{0.5in}
        \begin{tikzpicture}[scale=1.1]
         \draw[-{Stealth[scale=1.0]}](1.8,1.8) -- (0.2,0.2);
        \draw[-{Stealth[scale=1.0]}] (0,1.8) -- (0,0.2);
        \draw[-{Stealth[scale=1.0]}] (1.8,0)-- (0.2,0) node[pos=.5, below=6pt] {$\vec{D}_1'''=4312$};
        \draw[-{Stealth[scale=1.0]}] (2,0.2) -- (2,1.8);
        \draw[-{Stealth[scale=1.0]}] (1.8,2) -- (0.2,2);
        \draw[-{Stealth[scale=1.0]}] (1.8,0.2) -- (0.2,1.8);
        \foreach \Point/\PointLabel in {(0,0)/2, (2,0)/4, (2,2)/3, (0,2)/1}
            \draw[fill=black] \Point circle (0) node {$\PointLabel$};
         \end{tikzpicture}
\end{center}
They can be obtained by a sequence of flips:
$$
\vec{D}_1\xrightarrow{\text{flip at 1}}\vec{D}_1'\xrightarrow{\text{flip at 3}}\vec{D}_1''\xrightarrow{\text{flip at 4}}\vec{D}_1'''\xrightarrow{\text{flip at 2}}\vec{D}_1.
$$
Therefore it is easy to see that $[\vec{D}_1]$ can be identified with $[2431]=\{\,2431,1243,3124,4312\,\}$.
\end{example}

As transitivity turns a DAG into a poset, we now introduce the definition of toric transitivity for a DAG, which will turn the corresponding toric DAG into a toric poset.

A DAG $\vec{D}$ with vertex set $[n]$ is {\em toric transitive} if the existence of a directed path $i_1\to i_2\to\dots\to i_k$ and $i_1\to i_k$  implies the existence of $i_a\to i_b$ in $\vec{D}$ for all $1\leq a<b\leq k$. 
A toric DAG $[\vec{D}]$ is a {\em toric poset} if $\vec{D}'$ is toric transitive for some $\vec{D}'\in[\vec{D}]$, or equivalently from \cite{agrr:cqf} Proposition 3.10, for all representatives $\vec{D}'$. 
A toric poset is a {\em total cyclic order} if one (or equivalently, all, according to Proposition \ref{prop:cp}) 
representative is a total linear order. In this case, we usually use the corresponding cyclic permutation $[w]$ to denote the total cyclic order $[\vec{D}]$.

\begin{remark}
In this paper, we adopt the definition of toric posets from \cite{agrr:cqf}, which is not quite the same as it was originally defined in \cite{Dmr:tpo}, but they are essentially equivalent by \cite{Dmr:tpo}, Theorem 1.4.
\end{remark}

For two toric DAGs $[\vec{D}_1],[\vec{D}_2]$ on the same vertex set $[n]$, we say $[\vec{D_2}]$ {\em extends} $[\vec{D}_1]$ 
if there exist $\vec{D}_i'\in[\vec{D}_i']$ for  $i=1,2$ such that $\vec{D}_2'$ extends $\vec{D}_1'$. 
If, furthermore, $[\vec{D}_2]$ is a total cyclic order corresponding to the cyclic permutation $[w]$, we also say that $[w]$ torically extends $[\vec{D}_1]$. 
Let $\mathcal{L}^{\tor}([\vec{D}])$ denote the set of cyclic permutations $[w]$ which torically extend $[\vec{D}]$.

\begin{example}\label{ex:toric extension}
In the Example~\ref{DAG on [4]}, both $[\vec{D}_1]=[2431]$ and $[\vec{D}_2]=[2413]$ torically extend $[\vec{D}_3]$. Moreover, they are the only total cyclic orders that torically extend $[\vec{D}_3]$, namely $\mathcal{L}^{\tor}([\vec{D}_3])=\{[2431],[2413]\}$.

In fact, Figure~\ref{fig: D_3} lists all representatives of $[\vec{D}_3]$, and it is straightforward to check that every total linear order linearly extending some DAG in $[\vec{D}_3]$ is in either $[2431]$ or $[2413]$.
\begin{figure}
\begin{center}
        \begin{tikzpicture}[scale=1.1]
         \draw[-{Stealth[scale=1.0]}] (0.2,0.2) -- (1.8,1.8);
        \draw[-{Stealth[scale=1.0]}] (0,0.2) -- (0,1.8);
        \draw[-{Stealth[scale=1.0]}] (0.2,0) -- (1.8,0) node[pos=.5, below=6pt] {$\vec{D}_3$};
        \draw[-{Stealth[scale=1.0]}] (2,0.2) -- (2,1.8);
        \draw[-{Stealth[scale=1.0]}] (1.8,0.2) -- (0.2,1.8);
        \foreach \Point/\PointLabel in {(0,0)/2, (2,0)/4, (2,2)/3, (0,2)/1}
            \draw[fill=black] \Point circle (0) node {$\PointLabel$};
         \end{tikzpicture}
         \hspace{0.5in}
        \begin{tikzpicture}[scale=1.1]
        \draw[-{Stealth[scale=1.0]}] (0.2,0.2) -- (1.8,1.8);
        \draw[-{Stealth[scale=1.0]}] (0,1.8) -- (0,0.2);
        \draw[-{Stealth[scale=1.0]}] (0.2,0) -- (1.8,0) node[pos=.5, below=6pt] {$\vec{D}_3'$};
        \draw[-{Stealth[scale=1.0]}] (2,0.2) -- (2,1.8);
        \draw[-{Stealth[scale=1.0]}] (0.2,1.8) -- (1.8,0.2);
        \foreach \Point/\PointLabel in {(0,0)/2, (2,0)/4, (2,2)/3, (0,2)/1}
            \draw[fill=black] \Point circle (0) node {$\PointLabel$};
         \end{tikzpicture}
          \hspace{0.5in}
        \begin{tikzpicture}[scale=1.1]
         \draw[-{Stealth[scale=1.0]}](1.8,1.8) -- (0.2,0.2);
        \draw[-{Stealth[scale=1.0]}] (0,1.8) -- (0,0.2);
        \draw[-{Stealth[scale=1.0]}] (0.2,0) -- (1.8,0) node[pos=.5, below=6pt] {$\vec{D}_3''$};
        \draw[-{Stealth[scale=1.0]}] (2,1.8) -- (2,0.2);
        \draw[-{Stealth[scale=1.0]}] (0.2,1.8) -- (1.8,0.2);
        \foreach \Point/\PointLabel in {(0,0)/2, (2,0)/4, (2,2)/3, (0,2)/1}
            \draw[fill=black] \Point circle (0) node {$\PointLabel$};
         \end{tikzpicture}
         
         \hspace{0.5in}
        \begin{tikzpicture}[scale=1.1]
         \draw[-{Stealth[scale=1.0]}](1.8,1.8) -- (0.2,0.2);
        \draw[-{Stealth[scale=1.0]}] (0,1.8) -- (0,0.2);
        \draw[-{Stealth[scale=1.0]}] (1.8,0)-- (0.2,0) node[pos=.5, below=6pt] {$\vec{D}_3'''$};
        \draw[-{Stealth[scale=1.0]}] (2,0.2) -- (2,1.8);
        \draw[-{Stealth[scale=1.0]}] (1.8,0.2) -- (0.2,1.8);
        \foreach \Point/\PointLabel in {(0,0)/2, (2,0)/4, (2,2)/3, (0,2)/1}
            \draw[fill=black] \Point circle (0) node {$\PointLabel$};
         \end{tikzpicture}
        \hspace{0.5in}
        \begin{tikzpicture}[scale=1.1]
         \draw[-{Stealth[scale=1.0]}] (1.8,1.8) -- (0.2,0.2);
        \draw[-{Stealth[scale=1.0]}] (0,0.2) -- (0,1.8);
        \draw[-{Stealth[scale=1.0]}] (0.2,0) -- (1.8,0) node[pos=.5, below=6pt] {$\vec{D}_3''''$};
        \draw[-{Stealth[scale=1.0]}] (2,1.8) -- (2,0.2);
        \draw[-{Stealth[scale=1.0]}] (1.8,0.2) -- (0.2,1.8);
        \foreach \Point/\PointLabel in {(0,0)/2, (2,0)/4, (2,2)/3, (0,2)/1}
            \draw[fill=black] \Point circle (0) node {$\PointLabel$};
         \end{tikzpicture}
\end{center}
\caption{All representatives of $[\vec{D}_3]$}
    \label{fig: D_3}
\end{figure}
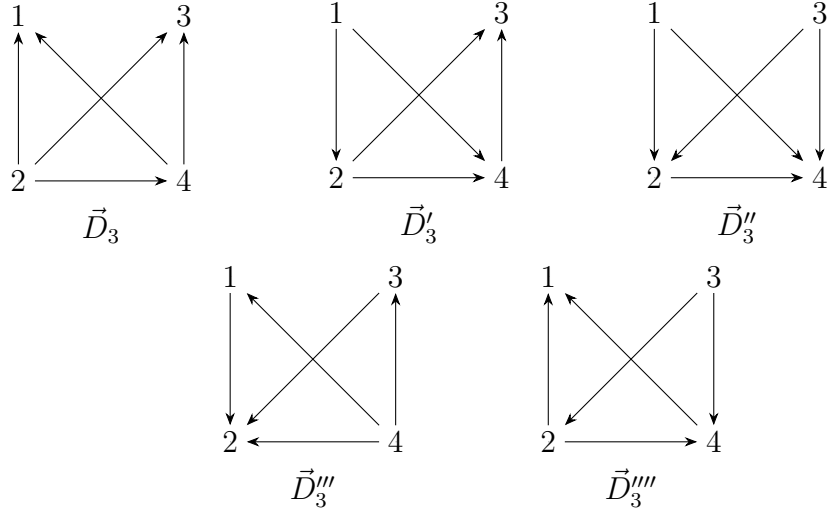
\end{example}

\begin{definition} [\text{Enriched toric $[\vec{D}]$-partition}]
An {\em enriched toric $[\vec{D}]$-partition} 
is a function $f:[n]\to \mathbb{P'}$ which is an enriched $\vec{D'}$-partition for at least one DAG $\vec {D}'$ in $[\vec{D}]$. Let $\mathcal{E}^{\tor}([\vec{D}])$ 
denote the set of all enriched toric $[\vec{D}]$-partitions.
\end{definition}

If $[\vec{D}]=[w]$ is a total cyclic order, the set of enriched toric $[w]$-partitions is the union of the set of enriched $w'$-partitions for all representatives $w'$ of $[w]$:
\begin{equation}\label{e-KK}
\mathcal{E}^{\tor}([w])=\bigcup_{w'\in[w]}\mathcal{E}(w').
\end{equation}

As in the linear case, we have the following  fundamental lemma for the decomposition of enriched toric $[\vec{D}]$-partitions. The proof is analogous to \cite{agrr:cqf}, Lemma 3.15.

\begin{lem} [\text{Fundamental lemma of enriched toric $[\vec{D}]$-partitions}]
\label{FLECDP}
For a DAG $\vec{D}$, the set of all enriched toric $[\vec{D}]$-partitions is a disjoint union of the set of enriched toric $[w]$-partitions of all toric extensions $[w]$ of $[\vec{D}]$:
$$\mathcal{E}^{\tor}([\vec{D}])=\bigsqcup_{[w]\in \mathcal{L}^{\tor}([\vec{D}])}\mathcal{E}^{tor}([w]).$$
\end{lem}

\begin{proof}
By the definition of $\mathcal{E}^{\tor}([\vec{D}])$, one has
\[
\mathcal{E}^{\tor}([\vec{D}])=\bigcup_{\vec{D'}\in[\vec{D}]}\mathcal{E}(\vec{D'}).
\]
In particular when $[\vec{D}]=[w]$ is a total cyclic order, it follows from the Proposition~\ref{prop:cp} that,
\[
\mathcal{E}^{\tor}([w])=\bigcup_{w'\in [ w]}\mathcal{E}({w'}).
\]
Hence,
\begin{align*}
    \mathcal{E}^{\tor}([\vec{D}])&=\bigcup_{\vec{D'}\in[\vec{D}]}\mathcal{E}(\vec{D'})\\
    & \overset{(i)}{=} \bigcup_{\vec{D'}\in[\vec{D}]}\bigcup_{w'\in\mathcal{L}(\vec{D'})}\mathcal{E}(w') \\
    & \overset{(ii)}{=} \bigcup_{[ w]\in\mathcal{L}^{\tor}([\vec{D}])} \bigcup_{{w'}\in[w]}\mathcal{E}(w')\\
    & = \bigcup_{[w]\in \mathcal{L}^{\tor}([\vec{D}])}\mathcal{E}^{\tor}([w]).
\end{align*}

To justify these steps, first note that equality $(i)$ follows from Lemma~\ref{FLEDP}.  

As for equality $(ii)$, it suffices to show that $w'\in\mathcal{L}(\vec{D'})$ for some $\vec{D'}\in[\vec{D}]$ if and only if $w'\in[w]$ for some $[w]\in\mathcal{L}^{\tor}([\vec{D}])$.
For the forward direction, if $w'\in\mathcal{L}(\vec{D'})$ for some $\vec{D'}\in[\vec{D}]$, then $[w']=[w]$ and $[w']$ torically extends $[\vec{D'}]=[\vec{D}]$. While for the reverse implication, given $w'\in[w]\in\mathcal{L}^{\tor}([\vec{D}])$, pick $\vec{D''}\in[\vec{D}]$ and $w''\in[w]$ with $w''$ linearly extending $\vec{D''}$, then $[{w''}]=[{w'}]=[w]$. It follows that there exists a sequence of flips which takes $w''$ to $w'$. Now applying the same sequence of flips on $\vec{D''}$ will result in some $\vec{D'}$. One then has $\vec{D'}\in[\vec{D''}]=[\vec{D}]$ and $w'\in \mathcal{L}(\vec{D'})$ as desired.

The assertion of disjointness follows directly from the fact that every function $f:[n]\to\mathbb{P}'$ has a unique linear permutation $w\in\mathcal{S}_n$ such that $f$ is also an enriched $w$-partition, hence an enriched toric $[w]$-partition. Such a linear permutation $w$ can be similarly constructed as in the proof of Lemma~\ref{FLEDP}, so the details are omitted. This completes the proof.
\end{proof}

For the sake of convenience, we always assume the label set $[n]$, however, it is note that the definition of (toric) DAGs, (toric) posets can be extended to any finite subsets of $\mathbb{P}$ as labels, with all consequent conclusions continuing to hold.


\section{Weight enumerators and algebra of cyclic peaks}\label{we}
In this section, we first review the weight enumerators for enriched $P$-partitions from Stembridge \cite{S:epp} in section~\ref{sec:P-partition}. These enumerators will span an algebra $\Pi$ referred as the algebra of peaks, which is a graded subring of $\QSym$. We then discuss its cyclic analogue for enriched toric $[\vec{D}]$-partitions in section~\ref{sec:D-partition}. In this case, enumerators will generate the algebra of cyclic peaks $\Lambda$, which is a graded subring of $\cQSym$. 

\subsection{Weight enumerator for enriched $\vec{D}$-partitions}\label{sec:P-partition}

Suppose $\vec{D}$ is a DAG on $[n]$. Define the {\em weight enumerator for enriched $\vec{D}$-partitions} 
by the formal series
$$\Delta_{\vec{D}}:=\sum_{f\in \mathcal{E}(\vec{D})}\prod_{i\in[n]}x_{|f(i)|},
$$
where $\mathcal{E}(\vec{D})$ is the set of enriched $\vec{D}$-partitions. By the Fundamental Lemma~\ref{FLEDP}, one has
\begin{equation}\label{eq:Delta_P summation}
    \Delta_{\vec{D}}=\sum_{w\in \mathcal{L}(\vec{D})}\Delta_w.
\end{equation}
It is clear from equation~\eqref{eq:enriched} that $\Delta_w$ is
a homogeneous quasi-symmetric function. More generally, $\Delta_{\vec{D}}$ is a homogeneous quasi-symmetric function in $\QSym$.

It also follows from equation~\eqref{eq:enriched} that the weight enumerator $\Delta_w$ depends on the descent set $\Des w$. A less obvious but important observation that $\Delta_w$ depends only on the peak set $\Pk w$ will follow directly from the next proposition, proved by Stembridge in \cite{S:epp}.

\begin{prop}[\cite{S:epp}, Proposition 2.2]
As a quasi-symmetric function, $\Delta_w$ has the following expansion of monomial quasi-symmetric functions
\begin{equation}\label{eq:M_E expansion}
    \Delta_w=\sum_{\substack{E\sbe[n-1]\colon\\[4pt] \Pk w\sbe E\cup(E+1)}}2^{|E|+1}M_{n,E},
\end{equation}
where the set $E+1$ is defined by~\eqref{eq:i+E}.
\end{prop}

\begin{example}\label{ex:enumerate toric}
From the Example~\ref{ex:toric extension}, we have $\mathcal{L}^{\tor}([\vec{D}_3])=\{\,[2431],[2413]\,\}$. By the Definition~\ref{def:enumerator} and equation \eqref{eq:cyclic enumerator decomp} in next subsection, one has
\begin{align*}
\Delta_{[\vec{D}_3]}^{\cyc}&=\Delta_{[2431]}^{\cyc}+\Delta_{[2413]}^{\cyc}\\
&=\Delta_{2431}+\Delta_{4312}+\Delta_{3124}+\Delta_{1243}+\Delta_{2413}+\Delta_{4132}+\Delta_{1324}+\Delta_{3241}.
\end{align*}

Applying the proposition above, we can calculate each summand as follows:
\begin{align*}
    \Delta_{2431}&=\sum_{E\sbe[3]\colon \{2\}\sbe E\cup(E+1)}2^{|E|+1}M_{4,E}\\[1ex]
    &=2^4M_{4,\{1,2,3\}}+2^3(M_{4,\{1,2\}}+M_{4,\{1,3\}}+M_{4,\{2,3\}})+2^2(M_{4,\{1\}}+M_{4,\{2\}}),\\[2ex]
    \Delta_{4312}&=\sum_{E\sbe[3]\colon \emptyset\sbe E\cup(E+1)}2^{|E|+1}M_{4,E}\\[1ex]
    &=2^4M_{4,\{1,2,3\}}+2^3(M_{4,\{1,2\}}+M_{4,\{1,3\}}+M_{4,\{2,3\}})+2^2(M_{4,\{1\}}+M_{4,\{2\}}+M_{4,\{3\}})+2M_{4,\emptyset},\\[2ex]
    \Delta_{3124}&=\sum_{E\sbe[3]\colon \emptyset \sbe E\cup(E+1)}2^{|E|+1}M_{4,E}\\[1ex]
    &=2^4M_{4,\{1,2,3\}}+2^3(M_{4,\{1,2\}}+M_{4,\{1,3\}}+M_{4,\{2,3\}})+2^2(M_{4,\{1\}}+M_{4,\{2\}}+M_{4,\{3\}})+2M_{4,\emptyset},\\[2ex]
    \Delta_{1243}&=\sum_{E\sbe[3]\colon \{3\}\sbe E\cup(E+1)}2^{|E|+1}M_{4,E}\\[1ex]
    &=2^4M_{4,\{1,2,3\}}+2^3(M_{4,\{1,2\}}+M_{4,\{1,3\}}+M_{4,\{2,3\}})+2^2(M_{4,\{2\}}+M_{4,\{3\}}),\\[2ex]  
    \Delta_{2413}&=\sum_{E\sbe[3]\colon \{2\}\sbe E\cup(E+1)}2^{|E|+1}M_{4,E}\\[1ex]
    &=2^4M_{4,\{1,2,3\}}+2^3(M_{4,\{1,2\}}+M_{4,\{1,3\}}+M_{4,\{2,3\}})+2^2(M_{4,\{1\}}+M_{4,\{2\}}),\\[2ex]
    \Delta_{4132}&=\sum_{E\sbe[3]\colon \{3\}\sbe E\cup(E+1)}2^{|E|+1}M_{4,E}\\[1ex]
    &=2^4M_{4,\{1,2,3\}}+2^3(M_{4,\{1,2\}}+M_{4,\{1,3\}}+M_{4,\{2,3\}})+2^2(M_{4,\{2\}}+M_{4,\{3\}}),\\[2ex]
    \Delta_{1324}&=\sum_{E\sbe[3]\colon \{2\}\sbe E\cup(E+1)}2^{|E|+1}M_{4,E}\\[1ex]
    &=2^4M_{4,\{1,2,3\}}+2^3(M_{4,\{1,2\}}+M_{4,\{1,3\}}+M_{4,\{2,3\}})+2^2(M_{4,\{1\}}+M_{4,\{2\}}),\\[2ex]
    \Delta_{3241}&=\sum_{E\sbe[3]\colon \{3\}\sbe E\cup(E+1)}2^{|E|+1}M_{4,E}\\[1ex]
    &=2^4M_{4,\{1,2,3\}}+2^3(M_{4,\{1,2\}}+M_{4,\{1,3\}}+M_{4,\{2,3\}})+2^2(M_{4,\{2\}}+M_{4,\{3\}}).
\end{align*}
Therefore, we have 
 \[   \arraycolsep=1.4pt\def\arraystretch{1.3}
    \begin{array}{rl}
    \Delta_{[\vec{D}_3]}^{\cyc}= 
    & 2^4\cdot 8M_{4,\{1,2,3\}}+2^3\cdot 8\left(M_{4,\{1,2\}}+M_{4,\{1,3\}}+M_{4,\{2,3\}}\right)+2^2\cdot 5\left(M_{4,\{1\}}+M_{4,\{3\}}\right)\\[1ex]
    &+2^2\cdot 8M_{4,\{2\}}+2\cdot 2M_{4,\emptyset}\\[2ex]
    =&2\cdot 2^4M^{cyc}_{(1,1,1,1)}+8\cdot 2^3M^{cyc}_{(2,1,1)}+5\cdot 2^2M^{cyc}_{(3,1)}+4\cdot 2^2M^{cyc}_{(2,2)}+2\cdot 2M^{cyc}_{(4)},\\
    \end{array} \]
where the second equality follows from the Table~\ref{table: cyclic M}.
\end{example}

As a counterpart, the weight enumerator $\Delta_w$ also has an expansion in terms of another basis: the fundamental quasi-symmetric functions.

\begin{prop}[\cite{S:epp}, Proposition 3.5]
As a quasi-symmetric function, $\Delta_w$ has the following expansion of fundamental quasi-symmetric functions
$$\Delta_w=2^{\pk w+1}\sum_{D\sbe[n-1]\colon \Pk w\sbe D\triangle (D+1)}F_{n,D}.
$$
Here $\triangle$ denotes the symmetric difference, that is, $D\triangle E=(D\cup E)\setminus(D\cap E)$.
\end{prop}

We say a subset $S\sbe[n]$ is a {\em peak set in $[n]$} if $\Pk w=S$ for some $w\in\mathcal{S}_n$. For any peak set $S$ in $[n]$, 
from the above proposition we can define an associated quasi-symmetric function by
\begin{equation*}
    K_S:=\Delta_w,
\end{equation*}
for any permutation $w$ with peak set $S$. It follows that $\Delta_w=K_{\Pk w}$ and one can rewrite equation~\eqref{eq:Delta_P summation} as
\[
\Delta_{\vec{D}}=\sum_{w\in \mathcal{L}(\vec{D})}K_{\Pk w}.
\]

Let $\Pi_n$ denote the space of quasi-symmetric functions spanned by $K_S$, taken over all peak sets in $[n]$ and set $\Pi=\oplus_{n\geq0}\Pi_n$. In \cite{S:epp} Stembridge referred to $\Pi$ as the algebra of peaks, and proved that $\Pi$ is a graded subring of $\QSym$. 

\subsection{Weight enumerator for enriched toric $[\vec{D}]$-partitions}\label{sec:D-partition}

\begin{definition}\label{def:enumerator}
For a given toric poset $[\vec{D}]$ with vertex set $[n]$, we define the {\em weight enumerator for enriched toric $[\vec{D}]$-partitions} by the formal power series 
$$\Delta_{[\vec{D}]}^{\cyc}:=\sum_{f\in \mathcal{E}^{\tor}([\vec{D}])}\prod_{i\in[n]}x_{|f(i)|}.$$
Namely, for integer $k>0$ we assign the weight $x_k$ to both $f$-values $k$ and $-k$.
\end{definition}

As a direct consequence of the Fundamental Lemma~\ref{FLECDP}, one has
\begin{equation}\label{eq:cyclic enumerator decomp}
    \Delta_{[\vec{D}]}^{\cyc}=\sum_{[ w]\in\mathcal{L}^{\tor}([\vec{D}])}\Delta_{[w]}^{\cyc}.
\end{equation}
Therefore, it suffices to discuss $\Delta_{[w]}^{\cyc}$ for cyclic permutations $[w]$. It follows from the formula \eqref{e-KK} that $\Delta_{[w]}^{\cyc}$ can be expressed in terms of the weight enumerators $\{\Delta_v\}$ as
\begin{equation}\label{e-Dsum}
    \Delta_{[w]}^{\cyc}=\sum_{v\in[w]}\Delta_v.
\end{equation}
Moreover, $\Delta_{[w]}^{\cyc}$ has the following expression.
\begin{prop}\label{prop:crucial}
For any given cyclic permutation $[w]$ of length $n$, we have 
\begin{equation}\label{eq:depend on cPk}
    \Delta_{[w]}^{cyc}=\sum_{\substack{E\subseteq[n]:\\[4pt]\cPk(w)\sbe E\cup(E+1)}}2^{|E|}M_{n,E}^{cyc},
\end{equation}
with set $E+1$ defined by~\eqref{eq:i+E}. 
The sum is independent on the choice of representative $w$ of $[w]$.
\end{prop}

\begin{proof}
The independence of representatives is a result of the following two observations:
\begin{enumerate}
    \item[(a)] If $E$ and $E'$ only differ by a cyclic shift, one has $|E|=|E'|$ and $M_{n,E}^{cyc}=M_{n,E'}^{cyc}$.
    \item[(b)] For two representatives $w$ and $w'$ of $[w]$, 
\[
\{E\subseteq[n]:\cPk(w)\sbe E\cup(E+1)\}=\{E'\sbe[n]:\cPk(w')\sbe E'\cup(E'+1)\}+i,
\]
for some $i\in[n]$, namely, the two sets only differ by a cyclic shift.
\end{enumerate}

Now fix a representative $w$ of $[w]$. We rewrite both sides of equation~\eqref{eq:depend on cPk} as follows:
\begin{equation}\label{eq:alpha_E}
    \mathrm{RHS}\overset{(i)}{=}\sum_{\substack{F\sbe[n]:\\[4pt]\cPk(w)\sbe F\cup(F+1)}}2^{|F|}\sum_{f\in F} M_{n,(F-f)\cap[n-1]}\overset{(ii)}{=}\sum_{E\sbe[n-1]}2^{|E|+1}\alpha_E M_{n,E}
\end{equation}
where $\alpha_E=\#A_E$, with 
$$A_E=\{(F,f): f\in F\sbe[n]\text{ with } \cPk(w)\sbe F\cup(F+1),\;E=(F-f)\cap[n-1]\}.$$ 
Here equality $(i)$ is a result of applying equation~\eqref{eq:monomial to cyclic monomial} to transform from cyclic monomial to monomial, while equality $(ii)$ is obtained by picking a $M_{n,E}$ and calculating the coefficient. It is noted that in equality $(ii)$, for each qualified pair $(F,f)\in A_E$, since $f\in F$, we have $n\in F-f$.
As a result, $E=(F-f)\cap[n-1]=(F-f)\setminus \{n\}$, or equivalently, $F=(E+f)\cup\{f\}$. Therefore their cardinalities satisfy $|F|=|E|+1$.
\begin{equation}\label{eq:beta_E}
    \mathrm{LHS}\overset{(i)'}{=}\sum_{v\in[w]}\Delta_v\overset{(ii)'}{=}\sum_{v\in[w]}\sum_{\substack{E\sbe[n-1]:\\[4pt] \Pk(v)\sbe E\cup(E+1)}}2^{|E|+1}M_{n,E}\overset{(iii)'}{=}\sum_{E\sbe[n-1]}\beta_E 2^{|E|+1}M_{n,E}
\end{equation}
where $\beta_E=\#B_E$ if we set $$B_E=\{i\in[n]: (\cPk w-i)\cap[2,n-1]\sbe E\cup(E+1)\}.$$
Equality $(i)'$ follows from equation~\eqref{e-Dsum}. Equality $(ii)'$ is obtained by applying equation~\eqref{eq:M_E expansion} and expressing the weight enumerators $\Delta_\pi$ in terms of monomial quasi-symmetric functions. 

As for equality $(iii)'$, notice that by interchanging the order of double summation, we have
\[
\sum_{v\in[w]}\sum_{\substack{E\sbe[n-1]:\\[4pt] \Pk(v)\sbe E\cup(E+1)}}2^{|E|+1}M_{n,E}=\sum_{E\sbe[n-1]}\gamma_E 2^{|E|+1}M_{n,E}
\]
where $\gamma_E=\#C_E$ if we set 
\[
C_E=\{v\in[w]: \Pk v\sbe E\cup(E+1)\}.
\]
From the observation 
\[
\{\{\,\Pk(v):v\in[w]\,\}\}=\{\{\,(\cPk w-i)\cap[2,n-1]:i\in[n]\,\}\}
\]
we obtain $\#C_E=\#B_E$ and hence $\gamma_E=\beta_E$. This proves that the equality $(iii)'$ holds.

By comparing equations~\eqref{eq:alpha_E} and \eqref{eq:beta_E}, it suffices to prove $\alpha_E=\beta_E$ for every $E\sbe[n-1]$, or equivalently, to construct a bijection between $A_E$ and $B_E$.

Set $\theta_E:A_E\to B_E$ as $\theta_E(F,f)=f$. To prove that this map is well-defined, it suffices to show that for each $(F,f)\in A_E$, we have $f\in B_E$. It follows from the definition of $A_E$ that $F-f=E\cup\{n\}$. Apply the operation on both sides of the equation $\cPk(w)\sbe F\cup(F+1)$, we get
$$\cPk(w)-f\sbe (F-f)\cup(F-f+1)=E\cup(E+1)\cup\{1,n\}.$$
Hence $(\cPk w-f)\cap[2,n-1]\sbe E\cup(E+1)$ and $f\in B_E$.

Conversely, define $\sigma_E:B_E\to A_E$ by $\sigma_E(f)=(\,(E+f)\cup\{f\},f\,)$. One can similarly check that this map is well-defined by the strategy of verifying $\theta_E$. 

It is straightforward to verify that $\theta_E$ and $\sigma_E$ are inverse, hence we get a bijection between $A_E$ and $B_E$. This finishes the proof.
\end{proof}

\begin{remark}
As a direct corollary of the above proposition, $\Delta_{[\vec{D}]}^{\cyc}$ is a homogeneous cyclic quasi-symmetric function of degree $|w|$, and that $\Delta_{[w]}^{\cyc}$ depends only on $\cPk[w]$, or equivalently from Remark~\ref{remark: cPk[w]}, $\cPk w$ for any representative $w$.
\end{remark}

\begin{example}
Let us compute an example to examine the validity of the previous proposition.  In particular, 
suppose $w=1243$, $[w]=\{1243,3124,4312,2431\}$ and $\cPk(w)=\{3\}$; $\pi=1324$, $[\pi]=\{1324,4132,2413,3241\}$ and $\cPk(\pi)=\{2,4\}$.

If we were to use the Proposition~\ref{prop:crucial} for calculation, we first need all possible choices of $E\subseteq[4]$ satisfying $\{3\}=\cPk w\sbe E\cup(E+1)$, which are
\begin{equation*}
    \{1,2,3,4\},\{1,2,3\},\{1,2,4\},\{1,3,4\},\{2,3,4\},\{1,2\},\{1,3\},\{2,3\},\{2,4\},\{3,4\},\{2\},\{3\},
\end{equation*}
with corresponding monomial cyclic quasi-symmetric functions
\[ \arraycolsep=1.4pt\def\arraystretch{1.4}
    \begin{array}{ccl}
     M^{cyc}_{(1,1,1,1)} &=&M^{cyc}_{4,\{1,2,3,4\}},  \\[1ex]
     M^{cyc}_{(2,1,1)} &=& M^{cyc}_{4,\{1,2,3\}}=M^{cyc}_{4,\{1,2,4\}}=M^{cyc}_{4,\{1,3,4\}}=M^{cyc}_{4,\{2,3,4\}}, \\[1ex]
     M^{cyc}_{(3,1)} &=&M^{cyc}_{4,\{1,2\}}=M^{cyc}_{4,\{2,3\}}=M^{cyc}_{4,\{3,4\}}, \\[1ex]
     M^{cyc}_{(2,2)} &=&M^{cyc}_{4,\{1,3\}}=M^{cyc}_{4,\{2,4\}}, \\[1ex]
     M^{cyc}_{(4)} &=&M^{cyc}_{4,\{2\}}=M^{cyc}_{4,\{3\}}. \\
    \end{array}
    \]
Applying equation~\eqref{eq:depend on cPk}, we have
\[
\Delta_{[w]}^{cyc}
    =2^4M^{cyc}_{(1,1,1,1)}+4\cdot 2^3M^{cyc}_{(2,1,1)}+3\cdot 2^2M^{cyc}_{(3,1)}+2\cdot 2^2M^{cyc}_{(2,2)}+2\cdot 2M^{cyc}_{(4)}.
\]

Similarly for the consideration of $\pi$, all possible choices of $E\subseteq[4]$ satisfying $\{2,4\}=\cPk \pi\sbe E\cup(E+1)$ are as follows:
\begin{equation*}
    \{1,2,3,4\},\{1,2,3\},\{1,2,4\},\{1,3,4\},\{2,3,4\},\{1,3\},\{1,4\},\{2,3\},\{2,4\},
\end{equation*}
hence by equation~\eqref{eq:depend on cPk},
\[
\Delta_{[\pi]}^{cyc}
    = 2^4M^{cyc}_{(1,1,1,1)}+4\cdot 2^3M^{cyc}_{(2,1,1)}+2\cdot 2^2M^{cyc}_{(3,1)}+2\cdot 2^2M^{cyc}_{(2,2)}.
\]
It follows from the calculation above that
\[
\Delta_{[w]}^{cyc}+\Delta_{[\pi]}^{cyc}=2\cdot 2^4M^{cyc}_{(1,1,1,1)}+8\cdot 2^3M^{cyc}_{(2,1,1)}+5\cdot 2^2M^{cyc}_{(3,1)}+4\cdot 2^2M^{cyc}_{(2,2)}+2\cdot 2M^{cyc}_{(4)},
\]
which is exactly the same enumeration as what we obtained in the Example~\ref{ex:enumerate toric}. This example verifies the validity of previous proposition.
\end{example}

\subsection{Algebra of cyclic peaks}

We say $S$ is a {\em cyclic peak set in $[n]$} if there exists some $w\in\mathcal{S}_n$ with $\cPk w=S$. It follows from Proposition~\ref{prop:crucial} that, for any cyclic peak set $S$ in $[n]$, we can define an associated cyclic quasi-symmetric function by
\begin{equation*}
    K^{\cyc}_S:=\Delta^{\cyc}_{[w]},
\end{equation*}
for any permutation $w$ with $\cPk w=S$. One can observe $\Delta^{\cyc}_{[w]}=K^{\cyc}_{\cPk w}$ and rewrite equation~\eqref{eq:cyclic enumerator decomp} as
\[
\Delta^{\cyc}_{[\vec{D}]}=\sum_{[w]\in \mathcal{L}^{\tor}([\vec{D}])}K^{\cyc}_{\cPk w}.
\]
Moreover, it follows from the formula \eqref{e-KK} that $K^{\cyc}_{S}$ can be expressed in terms of quasi-symmetric functions $\{K_T\}$ as
\[
K^{\cyc}_{S}=\sum_{\sigma\in[w]} K_{\Pk\sigma},
\]
where $\cPk w=S$.

Let $\Lambda_n$ denote the space of cyclic quasi-symmetric functions spanned by $K^{\cyc}_S$ where $S$ ranges over cyclic peak sets in $[n]$, and set $\Lambda=\oplus_{n\geq 0}\Lambda_n$. We now will show that $\Lambda$ is an algebra by proving that the product of $K^{\cyc}_U$ and $K^{\cyc}_T$ is a linear combination of $K^{\cyc}_S$'s, and we call $\Lambda$ the {\em algebra of cyclic peaks}. 

\begin{lem}
The $K^{\cyc}_S$ are linearly independent, where $S$ are distinct cyclic peak sets up to cyclic shift.
\end{lem}
\begin{proof}
In this proof, we totally order the subsets of $[n-1]$, 
first by cardinality, then by the lexicographic order. We therefore have
\[
\emptyset\lhd\{1\}\lhd\{2\}\lhd\cdots\lhd\{n-1\}\lhd\{1,2\}\lhd\{1,3\}\lhd\cdots\lhd\{1,n-1\}\lhd\{2,3\}\lhd\{2,4\}\lhd\cdots.
\]
One can similar order the compositions of $n$ as
 \[
(n)\lhd(1,n-1)\lhd(2,n-2)\lhd\cdots\lhd(1,1,n-2)\lhd(1,2,n-3)\lhd\cdots\lhd(1,n-2,1)\lhd\cdots.
\]

Notice that $K^{\cyc}_S=K^{\cyc}_{S'}$ if two sets $S$ and $S'$ only differ by a cyclic shift, we will always assume the index $S$ to be the least among all its cyclic shifts. 
It is not hard to see that $\psi(S)$ is also the least composition among all its cyclic shifts, where $\psi$ is defined by equation \eqref{eq:psi}. 

We now show that the matrix of $\{K^{\cyc}_S\}$ with respect to $\{M^{\cyc}_{n,L}\}$ is a matrix of full rank. 
Then the linear independence of $\{K^{\cyc}_S\}$ will follow immediately from the fact that the monomial cyclic quasi-symmetric functions form a basis of $\cQSym$.

Let us fix $n$ and suppose $\{S_1\lhd S_2\lhd\cdots\lhd S_m\}$ is the set of all distinct cyclic peak sets in $[n]$. Given a $K^{\cyc}_S$, suppose $|S|=k$ and $S=\{1=s_1<s_2<\cdots<s_k\}$. Here all indices are taken modulo $k$ unless otherwise noted. To each $K_S^{\cyc}$ we associate a corresponding monomial quasi-symmetric function by $F(K^{\cyc}_S)=M^{\cyc}_{n,f(S)}$, where
\[
f(S)=\{s_1,s_2-1,\cdots,s_k-1\}.
\]
Since $S$ is a cyclic peak set, elements in $f(S)$ are distinct. So $f(S)$ is a set and $F$ is well-defined. If one denotes $\psi(S)$ by $(\al_1,\al_2,\dots,\al_k)$, then 
\[
\psi(f(S))=(\al_1-1,\al_2,\dots,\al_{k-1},\al_k+1).
\]
Also notice that the assumption of $S$ being the least ensures that $f(S)$ is also the least among all its cyclic shifts. It follows that $F$ is injective.

\textbf{Claim:}  The matrix of $\{K^{\cyc}_S\}$ related to corresponding $\{M^{\cyc}_{n,f(S)}\}$ is invertible. Moreover, this matrix is an upper triangular matrix with nonzero diagonal entries.
\vspace{3pt}

Let $A_{i,j}$ denote the coefficient of $M^{\cyc}_{n,f(S_j)}$ in the expression of $K^{\cyc}_{S_i}$ and set $A=\{A_{i,j}\}$ to be the $m\times m$ matrix that we need to consider. To prove that $A$ is an upper triangular matrix, it suffices to show that $A_{i,j}=0$ if $i>j$, or equivalently, the term $M^{\cyc}_{n,f(S_j)}$ does not appear in the expression of $K^{\cyc}_{S_i}$ if $S_j\lhd S_i$. 

Suppose towards a contradiction that $A_{i,j}\neq0$ for some $i>j$. It then follows from the Proposition~\ref{prop:crucial}
that there exists some $E\subseteq[n]$ such that
\[
S_i\subseteq E\cup(E+1)\quad \text{and}\quad M^{\cyc}_{n,f(S_j)}=M^{\cyc}_{n,E}.
\]
Since $S_i$ is a cyclic peak set, $e$ and $e+1$ cannot be in $S_i$ at the same time. It then follows from the assumption $S_i\subseteq E\cup(E+1)$ that $|E|\geq |S_i|$. The other condition $i>j$ implies $S_i\rhd S_j$, hence $|S_i|\geq |S_j|=|f(S_j)|=|E|$. Therefore, we only need to consider the cases when $S_i$ and $S_j$ have the same cardinality. 

Suppose 
\[
\psi(S_i)=(\alpha_1,\al_2,\dots,\al_k),\quad \psi(S_j)=(\be_1,\be_2,\dots,\be_k),\quad \psi(E)=(\be_1',\be_2',\dots,\be_k').
\]
Then $\psi(E)=(\be_1',\be_2',\dots,\be_k')$ is a cyclic shift of $\psi(f(S_j))=(\be_1-1,\be_2,\dots,\be_{k-1},\be_k+1)$. Assume 
\[
(\be_q',\be_{q+1}',\dots,\be_{q-1}')=(\be_1-1,\be_2,\dots,\be_{k-1},\be_k+1),
\]
for some $q\in[k]$. Notice from the condition $S_i\subseteq E\cup(E+1)$, the cardinality relation $|S_i|=|E|$, and the fact that $S_i$ does not have consecutive elements as a cyclic peak set, each $a\in S_i$ corresponds to a unique $a'\in E$ where $a'$ equals $a$ or $a-1$ (considered modulo $n$). A crucial observation therefore follows:
\begin{align*}
    \text{Suppose $a,b\in S_i$ correspond to $a',b'\in E$ respectively, then $|(a-b)-(a'-b')|\leq 1$.}
\end{align*}
Equivalently, one has, for $r,s\in[k]$,
\begin{equation}\label{inequ3}
    |(\al_r+\cdots+\al_s)-(\be_r'+\cdots+\be_s')|\leq 1.
\end{equation}
In particular, 
\[\al_1\leq\al_q\leq\be_q'+1=(\be_1-1)+1=\be_1,\]
where the first inequality comes from the fact that $S_i$ is the least among its cyclic shifts, and the second one follows from~\eqref{inequ3} by taking  $r=s=q$.
Note that $S_j\lhd S_i$ implies $\psi(S_j)\lhd\psi(S_i)$, so $\be_1\leq \al_1$. Thus, it must be 
\[\al_1=\al_q=\be_q'+1=\be_1.\]
It then follows from the inequality \eqref{inequ3} that for any $r\in[k-1]$,
\[
 \al_q+\cdots+\al_{q+r}\leq \be_q'+\cdots+\be_{q+r}'+1=\be_1+\cdots+\be_{1+r}+\delta_{1+r,k},
\]
where $\delta_{1+r,k}=1$ if $1+r=k$ and $0$ otherwise.

If $\al_{q+r}=\be_{1+r}$ for every $r\in[k-2]$, then $\psi(S_i)$ is a cyclic shift of $\psi(S_j)$, 
which contradicts to our assumption $S_i\rhd S_j$.
Now assume $t\in[k-2]$ is the smallest index such that $\al_{q+t}<\be_{1+t}$, then 
\[
(\al_{q},\dots,\al_{q+t})\lhd(\be_1,\dots,\be_{1+t}).
\]
Combined with the inequality $(\al_1,\dots,\al_{1+t})\unlhd(\al_{q},\dots,\al_{q+t})$, as $(\al_1,\dots,\al_k)$ is the least among all its cyclic shifts, one has 
\[
(\al_1,\dots,\al_{1+t})\lhd(\be_1,\dots,\be_{1+t}).
\]
This implies that $\psi(S_i)\lhd\psi(S_j)$, which is a contradiction to the assumption $S_i\rhd S_j$.
Moreover, it is not hard to see that $A_{i,i}\neq 0$ as $S_i\subseteq f(S_i)\cup(f(S_i)+1)$, therefore the diagonal entries are nonzero.
This proves the claim, hence the lemma.
\end{proof}

Define {\em the union of digraphs} $\vec{D}\biguplus\vec{E}$ to be the digraph with vertices $V(\vec{D})\cup V(\vec{E})$ and arcs $A(\vec{D})\cup A(\vec{E})$. Next result follows easily from the definition.
\begin{prop}
Suppose $\vec{D}$ and $\vec{E}$ are two DAGs on disjoint subsets of $\mathbb{P}$, then
\begin{equation}\label{eq:joint union of DAGs}
    \Delta^{\cyc}_{[\vec{D}\biguplus\vec{E}]}=\Delta^{\cyc}_{[\vec{D}]}\cdot\Delta^{\cyc}_{[\vec{E}]}.
\end{equation}
\end{prop}

The proposition above provides us a proof of $\Lambda$ being an algebra from the combinatorial aspect.

\begin{prop}
$\Lambda$ is a graded subring of $\cQSym$.
\end{prop}
\begin{proof}
As a subspace of $\cQSym$, $\Lambda$ naturally inherits the addition and multiplication operations from $\cQSym$.

To show that $\Lambda$ is closed under multiplication, take $K^{\cyc}_U\in\Lambda_m$ and $K^{\cyc}_T\in\Lambda_n$, where $U$ and $T$ are cyclic peak sets in $[m]$ and $[n]$ respectively, namely, there exist $\pi\in\mathcal{S}_m$ and $w\in\mathcal{S}_n$ such that $\cPk\pi=U,\;\cPk w=T$. For the purpose of constructing two corresponding disjoint DAGs, we standardize $w=w_1w_2\ldots w_n\in\mathcal{S}_n$ to $\{m+1,m+2,\ldots,m+n\}$, that is, construct $w'=w'_1w_2'\ldots w_n'$ where $w_i'=w_i+m$ for $i\in[n]$. As a consequence of equation~\eqref{eq:joint union of DAGs}, we have
\begin{equation}
    K^{\cyc}_U\cdot K^{\cyc}_T=\Delta^{\cyc}_{[\pi]}\cdot\Delta^{\cyc}_{[w']}=\Delta^{\cyc}_{[\pi]\biguplus[w']}=\sum_{\sigma\in \mathcal{L}^{\tor}([\pi]\biguplus[w'])}K^{\cyc}_{\cPk \sigma}.
\end{equation}
This completes the proof.
\end{proof}

In terms of another basis of $\cQSym$, the fundamental cyclic quasi-symmetric functions,  $K^{\cyc}_S$ has the following expansion.

\begin{prop}\label{expand K in F}
For any cyclic peak set $S$ in $[n]$, we have
\[ 
2^{-|S|}K^{\cyc}_S=\sum_{\substack{E\subseteq[n]:\\[2pt]S\sbe E\triangle (E+1)}}F^{cyc}_{n,E},
\]
where $\triangle$ denotes symmetric difference.
\end{prop}
\begin{proof}
Since $F^{cyc}_{n,E}=\sum_{F\supseteq E} M^{cyc}_{n,F}$, the coefficient of $M^{cyc}_{n,F}$ on the right side of the above expansion is $|\{E\subseteq F: S\sbe E\triangle (E+1)\}|$. 

To count this set, we need the following observations for possible choices of $E$. For each $k\in S\sbe E\triangle (E+1)$, exactly one of $k-1$ and $k$ is in $E$, since if both are in $E$ then $k$ is not in the symmetric difference. It follows from $E\sbe F$ that at least one of $k-1$ and $k$ is in $F$.
So one has the following two cases:
\begin{enumerate}
    \item[(1)] If both $k,k-1\in F$, $E$ must contain exactly one of $k$ or $k-1$. Otherwise, suppose $E$ contains both $k$ and $k-1$, then $k\notin E\triangle (E+1)$ by the definition of symmetric difference. This is a contradiction to the assumption $k\in S\sbe E\triangle (E+1)$. 
    \item[(2)] If only one of $k,k-1\in F$, then $E$ must contain whichever $F$ does.
\end{enumerate}
It is noted that the restriction above only involve two adjacent numbers.

Also notice that if $k\in F$ but neither $k$ nor $k+1$ is in $S$, then $k$ is free to be in $E$ or not. Denote
\begin{align*}
    S_1&=\{k\in S\mid \text{both $k,k-1$ are in $F$}\},\\
    S_2&=\{k\in S\mid k\in E,k-1\notin F\},\\
    S_3&=\{k\in S\mid k\notin E,k-1\in F\}.
\end{align*}
Then we have a set partition
$S=S_1\cup S_2\cup S_3$. Therefore if we denote $s_i=\#S_i$ for $i\in\{1,2,3\}$, we have $|S|=s_1+s_2+s_3$.

Moreover, if we denote 
\[
T=\{k\in F\mid \text{none of $k,k+1$ is in $S$}\},
\]
notice from the definition of peak set $S$, numbers in $S$ are not adjacent. Hence we have the following set partition 
of $F$:
\[
F=S_1\cup(S_1-1)\cup S_2\cup (S_3-1)\cup T.
\]
It follows that $|F|=2s_1+s_2+s_3+t$ with $t=|T|$. Hence, the number of choices for $E$ is 
\[
2^{s_1+t}=2^{s_1+|F|-2s_1-s_2-s_3}=2^{|F|-(s_1+s_2+s_3)}=2^{|F|-|S|}.
\]
So we have
\[
\sum\limits_{\substack{E\subseteq[n]:\\[2pt]S\sbe E\triangle(E+1)}}F^{cyc}_{n,E}=\sum\limits_{\substack{F\subseteq[n]:\\[2pt]S\sbe F\cup(F+1)}}2^{|F|-|S|}M^{cyc}_{n,F}.
\]
By equation~\eqref{eq:depend on cPk}, this quantity is $2^{-|S|}K^{cyc}_S$.
\end{proof}

\subsection{Order polynomials $\Omega^{\cyc}([\vec{D}],m)$}
Given a DAG $\vec{D}$, we can define the {\em order polynomial of enriched $\vec{D}$-partitions}, $\Omega(\vec{D},m)$, by
\[
\Omega(\vec{D},m)=\Delta_{\vec{D}}(1^m),
\]
where $\Delta_{\vec{D}}(1^m)$ means that we set $x_1=\cdots =x_m=1$, and $x_k=0$ for $k>m$. In fact, $\Omega(\vec{D},m)$ counts the number of enriched $\vec{D}$-partitions with absolute value at most $m$.

Moreover, Stembridge computed the corresponding generating function as follows:

\begin{thm}[\cite{S:epp}, Theorem 4.1]\label{t-opstem}
For a given $w\in\mathcal{S}_n$, one has
\begin{equation}\label{eq:WEEP}
    \sum_m\Omega(w,m)t^m=\frac{1}{2}\left(\frac{1+t}{1-t}\right)^{n+1}\left(\frac{4t}{(1+t)^2}\right)^{1+\pk w}.
\end{equation}
\end{thm}

It is not hard to see that $\Omega(\vec{D},t)$ is indeed a polynomial in $t$.
Suppose $\vec{D}$ has vertex set $V$ and that $|V|=n$, then
\[
\Omega(\vec{D},m)=\sum_{k=1}^{n}c_k\binom{m}{k},
\]
where $c_k$ denotes the number of $f\in \mathcal{E}(\vec{D})$ such that $\{|f(x)|\colon x\in V\}=[k]$. For any fixed $k$, $\binom{m}{k}$ is a polynomial in $m$ of degree $k$. Since $c_k$ and $n$ are constants, it follows that the summation is also a polynomial in $m$. This verifies that $\Omega([\vec{D}],t)$ is a polynomial.

Similarly, we define the {\em order polynomial of enriched toric $[\vec{D}]$-partitions}, $\Omega^{\cyc}([\vec{D}],m)$, by
\[
\Omega^{\cyc}([\vec{D}],m)=\Delta^{\cyc}_{[\vec{D}]}(1^m).
\]

The following result is the toric analogue of formula \eqref{eq:WEEP}.
\begin{prop}\label{op}
Given $w\in \mathcal{S}_n$, then
\begin{equation}\label{e-Ocyc}
    \sum_m\Omega^{\cyc}([w],m)t^m=\left(\frac{4t}{(1+t)^2}\right)^{\cpk w}\left(\frac{1+t}{1-t}\right)^{n-1}\left(\cpk w+\frac{2nt}{(1-t)^2}\right).
\end{equation}
This right side of the equation does not depend on the choice of representative $w$, as they all have the same cyclic peak number.
\end{prop}
\begin{proof}
By the definition of order polynomial,
\begin{align*}
    \sum_m\Omega^{\cyc}([w],m)t^m &=\sum_m\Delta^{\cyc}_{[w]}(1^m)t^m\\
    &=\sum_m\sum_{v\in[w]}\Delta_{v}(1^m)t^m\\
    &=\sum_{v\in[w]}\sum_m\Delta_{v}(1^m)t^m\\
    &=\sum_{v\in[w]}\sum_m\Omega(v,m)t^m\\
    &=\sum_{v\in[w]}\frac{1}{2}\left(\frac{1+t}{1-t}\right)^{n+1}\left(\frac{4t}{(1+t)^2}\right)^{1+\pk v}\\
    &=\frac{1}{2}\left(\frac{1+t}{1-t}\right)^{n+1}\sum_{v\in[w]}\left(\frac{4t}{(1+t)^2}\right)^{1+\pk v},
\end{align*}
where the last equality is obtained by applying \eqref{eq:WEEP}.

Observe that each representative of $[w]$ will start with a peak, end with a peak or none of the two ends are peaks, which will result in the peak number $\cpk w-1$, $\cpk w-1$ or $\cpk w$ respectively. And the number of those representatives with a peak element at one end is $2\cpk w$. It follows from the discussion above that
\begin{align*}
\sum_m\Omega^{\cyc}([w],m)t^m
&=\frac{2\cpk w}{2}\left(\frac{1+t}{1-t}\right)^{n+1}\left(\frac{4t}{(1+t)^2}\right)^{\cpk w}+\frac{n-2\cpk w}{2}\left(\frac{1+t}{1-t}\right)^{n+1}\left(\frac{4t}{(1+t)^2}\right)^{1+\cpk w}\\
&=\left(\frac{4t}{(1+t)^2}\right)^{\cpk w}\left(\frac{1+t}{1-t}\right)^{n-1}\left(\cpk w +\frac{2nt}{(1-t)^2}\right).    
\end{align*}
This completes the proof.
\end{proof}

Notice that by taking the coefficient of $t^m$ on both sides of equation \eqref{e-Ocyc}, one can get the expression of the order polynomial of enriched $\vec{D}$-partitions $\Omega^{\cyc}([w],m)$ in an algebraic manner. It would be desirable to derive $\Omega^{\cyc}([w],m)$ combinatorially. For this purpose, we first give a combinatorial proof for the order polynomial in the linear case. We will define the $(w,m)$-marking on the permutation $w$ with a given positive integer $m$. Then we will show that each $(w,m)$-marking corresponds to exactly $2^{2\pk +1}$ enriched $w$-partitions with absolute value at most $m$. 

Suppose $w\in\mathcal{S}_n$, $m\in\mathbb{P}$. One can naturally extend $w=w_1\ldots w_n$ to $w'=w_0 w_1\ldots w_n w_{n+1}$ where $w_0=w_{n+1}=\infty$. Let $R_1$ 
be the longest decreasing initial factor of $w'$.
Now let $v$ denote $w'$ with $R_1$ deleted and let $R_2$ be the longest increasing initial factor of $w'$. Continue in this way and alternate decreasing and increasing factors, we can get a factorization of $w'$. We will call the factors {\em runs} and the corresponding indices {\em run indices}. We denote by $I_j$ the set of run indices corresponding to factor set $R_j$. It is noted that any extension $w'$ will start with a decreasing run and end with an increasing run, which implies that the number of runs is always even.

Take $$w=\begin{array}{cccccc}
     1&2&3&4&5&6  \\
     \textbf{1}&4&{\bf3}&{\bf2}&5&6 
\end{array}$$ as an example. The corresponding natural extension
$$w'=\begin{array}{cccccccc}
     0&1&2&3&4&5&6&7  \\
     \textbf{$\infty$}&\textbf{1}&4&{\bf3}&{\bf2}&5&6&\text{$\infty$}
\end{array}$$
has four runs
\[
R_1=\infty1,\;R_2=4,\; R_3=32,\;R_4=56\infty,
\]
where the decreasing runs are in bold, and corresponding set of run indices are
\[
I_1=\{0,1\},\,\;I_2=\{2\},\; I_3=\{3,4\},\;I_4=\{5,6,7\}.
\]

Suppose permutation $w'$ has $r$ runs. We have the following observations:
\begin{enumerate}
    \item The parity of $i$ indicates the trending of each run $R_i$. If $i$ is even, $R_i$ is increasing. If $i$ is odd, then $R_i$ is decreasing.
    \item The total number of runs $r$ is closely related to the peak number $\pk w$:
    $$r=2\pk w+2.$$
\end{enumerate}

We will decorate permutations with bars and marks. One can put bars between adjacent columns in two-line notation (including the space before the first column and the space after the last), whereas a column of $w$ with index $i\in[n]$ can be marked if and only if $i,i+1\in I_j$ for some $j$. In other word, $i$ and $i+1$ are in the same run index set. There can be multiple bars between two adjacent columns and we count them with multiplicity, while each column can be marked at most once. We will denote by $\cM_w$ 
the set of indices where the corresponding columns can be marked,
\[
\cM_w:=\{i\in[n]\mid \text{ $i,i+1\in I_j$ for some $j$}\}.
\]
We note that for a given $w$, the cardinality of the complement of set $\cM_w$ in $[n]$ is $2\pk w+1$. Equivalently, one has $|\cM_w|=n-2\pk w-1$.

\begin{definition}[\text{$(w,m)$-marking}]
Suppose $w$ is a linear permutation, $m$ is a positive integer. A {\em $(w,m)$-marking} is a marking of $w$ using $b$ bars and $d$ marked columns, satisfying that $b+d=m-1-\pk w$.
\end{definition}

\begin{example}\label{ex:markings}
If we set $w=143256$ as usual, and take $m=5$, then a $(w,5)$-marking has $b$ bars and $d$ marked columns, such that $b+d=3$. Two $(w,5)$-markings are provided as follows. Both have two bars and one marked column, where the marked column is in blue.

$$\begin{array}{cc||cccc}
     1&2&3&4&\textcolor{blue}{\bf5}&6  \\
     1&4&3&2&\textcolor{blue}{\bf5}&6 
\end{array}\qquad\qquad
\begin{array}{cc|cccc|}
     1&2&\textcolor{blue}{\bf3}&4&5&6  \\
     1&4&\textcolor{blue}{\bf3}&2&5&6 
\end{array}
$$
\end{example}

\begin{prop}\label{p-orderpoly}
For a given $w\in\mathcal{S}_n$, one has
\begin{equation}\label{eq:OPEDP}
    \Omega(w,m)=2^{2\pk w+1}\sum_{k=0}^{m-1-\pk w}\left(\binom{n+1}{k}\right)\binom{n-2\pk w-1}{m-1-\pk w-k},
\end{equation}
where $(\binom{n+1}{k})$ denotes the number of multisets on $[n+1]$ with cardinality $k$.
\end{prop}

\begin{proof}
It is clear by definition that the number of possible choices for $(w,m)$-markings is
\[
\sum_{b+d=m-1-\pk w}\left(\binom{n+1}{b}\right)\binom{n-2\pk w-1}{d}=\sum_{k=0}^{m-1-\pk w}\left(\binom{n+1}{k}\right)\binom{n-2\pk w-1}{m-1-\pk w-k}.
\]
Therefore, it suffices to construct a $2^{2\pk w+1}$-to-one map from the set of enriched $w$-partitions with absolute value at most $m$ to the set of all $(w,m)$-markings.

Given an enriched $w$-partition $f$ with absolute value at most $m$, we can inductively associate it with a unique $(w,m)$-marking as follows:

We first determine whether column $k$ gets marked for those $k\in \cM_w$. Suppose $k\in I_i$ for some $i\in[r]$. We mark column $k$ if and only if $\delta_k+\gamma_k=1$ where 
$$
\delta_k:=\delta\,(\text{$i$ is even and $f(w_{k})<0$}),\qquad \gamma_k:=\delta\,(\text{$i$ is odd and $f(w_{k})>0$}).
$$
Here the {\em Kronecker function} on a statement $R$ is defined by
\[
\delta(R)=\begin{cases} 
      1 & \text{if $R$ is true},\\
     0 & \text{if $R$ is false}.
   \end{cases}.
\]

As for the placement of bars, we start by putting $|f(w_1)|-1$ bars before the first column. Inductively, suppose $k\in I_i$ for some $k\in[2,n]$, $i\in[r]$ and that we have already constructed the markings and bars on and before the $(k-1)$st column. Then the number of marks and bars strictly before the $k$th column is constructed to be 
\begin{equation}
\label{mb}
    |f(w_{k})|-\lceil  i/2\rceil -\delta_k.
\end{equation}
Finally we add bars after the last column so that the total number of marks and bars is $m-\pk w-1$. In this manner, we can inductively define a unique $(w,m)$-marking for $f$.

We must verify that the constructed marking is indeed a $(w,m)$-marking. Firstly we will verify that the construction before the $n$th column is possible in that \eqref{mb} is a weakly increasing function of $k$ and strictly increasing from the $k$th to the $(k+1)$st term if column $k$ is marked.
In other words, if $k\in I_i$ and $k+1\in I_j$, it suffices to show
\[
    |f(w_{k})|-\lceil i/2\rceil -\delta_{k}\leq |f(w_{k+1})|-\lceil j/2\rceil -\delta_{k+1},
\]
for all $k$, and that
\[
    |f(w_{k})|-\lceil i/2\rceil -\delta_{k}< |f(w_{k+1})|-\lceil j/2\rceil -\delta_{k+1},
\]
when the column $k$ is marked. Notice that
\[(\delta_{k}+\gamma_{k})\cdot\delta(k\in \cM_w)=1\]
if column $k$ is marked and $0$ otherwise. Hence it suffices to show that
\begin{equation}\label{inequ}
    |f(w_{k})|-\lceil i/2\rceil -\delta_{k}+(\delta_{k}+\gamma_{k})\cdot\delta(k\in \cM_w)\leq |f(w_{k+1})|-\lceil j/2\rceil -\delta_{k+1}.
\end{equation}
By the definition of enriched $P$-partitions, we have 
\begin{equation}\label{reduced inequ}
    |f(w_{k})|\leq |f(w_{k+1})|.
\end{equation}
Let us consider the following cases:
\begin{enumerate}
    \item[(a)] If $j=i$, it follows that $k\in \cM_w$. Hence the inequality \eqref{inequ} simplifies to prove
    \[
|f(w_{k})|+\gamma_{k}\leq |f(w_{k+1})|-\delta_{k+1}.
\]
     If $i$ is even, then $\gamma_{k}=0$ and $w_{k}<w_{k+1}$.
     By the inequality~\eqref{reduced inequ}, one only needs to consider whether the inequality holds when $\delta_{k+1}=1$, which implies $f(w_{k+1})<0$. It follows from the definition of enriched $P$-partitions that $f(w_{k+1})\succeq f(w_{k})-1$, or equivalently, $|f(w_{k+1})|\geq |f(w_{k})|+1$, which proves \eqref{inequ} in this case. The proof is similar when $i$ is odd.
     \item[(b)] If $j=i+1$, then $k\notin \cM_w$. The inequality \eqref{inequ} reduces to
 \[
|f(w_{k})|-\lceil i/2\rceil-\delta_{k}\leq |f(w_{k+1})| -\lceil  (i+1)/2\rceil-\delta_{k+1}.
\]
     If $i$ is even, then $w_{k}<w_{k+1}$, $\lceil i/2\rceil+1=\lceil (i+1)/2\rceil$ and $j$ is odd, hence $\delta_{k+1}=0$. Therefore one only needs to prove
  \[
|f(w_{k})|-\delta_{k}\leq |f(w_{k+1})| -1.
\]    
     
     The inequality clearly holds when $\delta_{k}=1$. So it suffices to consider the case when $\delta_{k}=0$. In this case, $f(w_{k})>0$, hence $f(w_{k+1})\succeq -(f(w_{k})+1)$ from the definition of enriched $P$-partitions, or equivalently $|f(w_{k+1})|\geq |f(w_{k})|+1$, which proves \eqref{inequ}. The case when $i$ is odd is similar and left to the readers.
\end{enumerate}

Secondly, one also needs to check that it is possible to add bars (possibly $0$) after the last column so that the total number of marks and bars is $m-\pk w-1$. In other words, the number of bars we add after the $n$th column is nonnegative. Notice that $\lceil \frac{i}{2}\rceil-1$ counts the number of peaks appear before the $k$-th column. In particular, $n\in I_r$ and from previous observation one has $\lceil \frac{r}{2}\rceil=\pk w+1$. By \eqref{mb}
the total number of bars and marked columns before the $n$th column is $|f(w_n)|-\pk w-1-\delta_n$. 
Together with the fact that the $n$th column is marked if and only if $\delta_n+\gamma_n=1$ and $n\in \cM_w$, the total number of bars that should be added after the $n$th column is 
\[
m-\pk w-1-(|f(w_n)|-\pk w-1)-\delta_n+(\delta_n+\gamma_n)\cdot\delta(n\in \cM_w)=m-|f(w_n)|+\delta_n-(\delta_n+\gamma_n)\cdot\delta(n\in \cM_w).
\]
Hence the nonnegativity condition becomes the following inequality,
\[
m-|f(w_n)|+\delta_n-(\delta_n+\gamma_n)\cdot\delta(n\in \cM_w)\geq 0,
\]
or equivalently,
\begin{equation}\label{inequ2}
m-|f(w_n)|\geq(\delta_n+\gamma_n)\cdot\delta(n\in \cM_w)- \delta_n.
\end{equation}

By assumption $f$ has absolute value at most $m$, hence $m-|f(w_n)|\geq 0$. Therefore, one only needs to consider whether the inequality holds when $(\delta_n+\gamma_n)\cdot\delta(n\in \cM_w)- \delta_n=1$, in which case it means $\delta_n+\gamma_n=1$, $\delta(n\in \cM_w)=1$ and $\delta_n=0$, or equivalently, $n\in I_i$ where $i$ is odd and $n\in \cM_w$. This is a contradiction as $n+1$ must be in a run where the index is even by definition, which implies that $n$ and $n+1$ cannot be in the same run index set, hence $n\notin \cM_w$. Now it is clear that $(\delta_n+\gamma_n)\cdot\delta(n\in \cM_w)- \delta_n\leq 0$, which completes the proof of inequality \eqref{inequ2}.
It therefore follows that the marking we constructed is a $(w,m)$-marking.

Now we need to show that for a given $(w,m)$-marking, there are $2^{2\pk w+1}$ different associated functions. From the above construction we notice that for those $k\in \cM_w$, $f(w_k)$ is uniquely determined. More precisely, whether column $k$ gets marked determines $\delta_k$ and $\gamma_k$ as $k$ determines the parity of $i$, hence the sign of $f(w_k)$ by the definition of $\delta_k$ and $\gamma_k$.
Suppose $k\in I_i$. The number of marks and bars strictly before the $k$th column determines the value $ |f(w_{k})|-\lceil  i/2\rceil -\delta_k$, therefore it determines $f(w_{k})$ as well. The only ambiguity is about $f(w_k)$ for those $k\notin \cM_w$. As the number of marks and bars strictly before the $k$th column fixes the value $ |f(w_{k})|-\lceil  i/2\rceil -\delta_k$, assuming to be $L$, there are two possible choices of the value $f(w_k)$ for each $k\notin \cM_w$: if $i$ is even, then $f(w_k)=-(L+\lceil  i/2\rceil+1)$ or $f(w_k)=L+\lceil  i/2\rceil$; if $i$ is odd, then $f(w_k)=L+\lceil  i/2\rceil$ or $f(w_k)=-(L+\lceil  i/2\rceil)$.
It follows that there are $2^{2\pk w+1}$ different functions corresponding to a given $(w,m)$-marking.
\end{proof}

\begin{example}
Still consider the permutation $w=143256$, and an enriched $w$-partition $f$ defined as follows:
\[
f(1)=1,\;f(4)=-2,\;f(3)=-4,\;f(2)=-4,\;f(5)=-5,\;f(6)=5.
\]
The corresponding marking is the first one in the Example~\ref{ex:markings}.
\end{example}

We now give a totally combinatorial derivation of the order polynomial $\Omega^{\cyc}([w],m)$ by using the Proposition~\ref{p-orderpoly}, which is also proved in a combinatorial way.

\begin{cor}
For a given $[w]\in[\mathcal{S}_n]$, one has
\begin{align*}
    \Omega^{\cyc}([w],m)&=(n-2\cpk w)\cdot2^{2\cpk w+1}\sum_{k=0}^{m-1-\cpk w}\left(\binom{n+1}{k}\right)\binom{n-2\cpk w-1}{m-1-\cpk w-k}\\
    &\qquad+\cpk w\cdot2^{2\cpk w}\sum_{k=0}^{m-\cpk w}\left(\binom{n+1}{k}\right)\binom{n-2\cpk w+1}{m-\cpk w-k}.
\end{align*}
This right side of the equation does not depend on the choice of representative $w$, as they all have the same cyclic peak number.
\end{cor}

\begin{proof}
Notice that any representative $w'$ of $[w]$ satisfies $\pk w'=\cpk [w]-1$ if $w'$ starts or ends with a cyclic peak. Otherwise, $\pk w'=\cpk[w]$. Among the $n$ representatives of $[w]$, there are $2\cpk [w]$ of them with a peak element at one end. Therefore, applying equation \eqref{e-Dsum} we have
\[
    \Omega^{\cyc}([w],m)=\Delta^{\cyc}_{[w]}(1^m)=\sum_{v\in[w]}\Delta_v(1^m)=\sum_{v\in[w]}\Omega(v,m),
\]
and by the previous observation and the Proposition~\ref{p-orderpoly}, one has
\begin{align*}
    \sum_{v\in[w]}\Omega(v,m)&=\sum_{v\in[w]}2^{2\pk v+1}\sum_{k=0}^{m-1-\pk v}\left(\binom{n+1}{k}\right)\binom{n-2\pk v-1}{m-1-\pk v-k}\\
    &=(n-2\cpk[w])\cdot 2^{2\cpk [w]+1}\sum_{k=0}^{m-1-\cpk [w]}\left(\binom{n+1}{k}\right)\binom{n-2\cpk [w]-1}{m-1-\cpk [w]-k}\\
    &\quad +2\cpk[w]\cdot2^{2(\cpk [w]-1)+1}\sum_{k=0}^{m-1-(\cpk [w]-1)}\left(\binom{n+1}{k}\right)\binom{n-2(\cpk [w]-1)-1}{m-1-(\cpk [w]-1)-k}\\
    &=(n-2\cpk w)\cdot2^{2\cpk w+1}\sum_{k=0}^{m-1-\cpk w}\left(\binom{n+1}{k}\right)\binom{n-2\cpk w-1}{m-1-\cpk w-k}\\
    &\qquad+\cpk w\cdot2^{2\cpk w}\sum_{k=0}^{m-\cpk w}\left(\binom{n+1}{k}\right)\binom{n-2\cpk w+1}{m-\cpk w-k}
\end{align*}
The conclusion follows immediately.
\end{proof}

We note that the generating function of order polynomials in Proposition~\ref{op} can also be deduced from the corollary above. More precisely,
\begin{align*}
    \sum_m\Omega^{\cyc}([w],m)t^m &=(n-2\cpk w)\cdot 2^{2\cpk w+1}\,\frac{(1+t)^{n-2\cpk w-1}}{(1-t)^{n+1}}\,t^{\cpk w+1}\\
    &\quad +\cpk w\cdot 2^{2\cpk w}\,\frac{(1+t)^{n-2\cpk w+1}}{(1-t)^{n+1}}\,t^{\cpk w}\\
    &=\left(\frac{4t}{(1+t)^2}\right)^{\cpk w}\left(\frac{1+t}{1-t}\right)^{n+1}\left((n-2\cpk w)\cdot\frac{2t}{(1+t)^2}+\cpk w\right)\\
    &=\left(\frac{4t}{(1+t)^2}\right)^{\cpk w}\left(\frac{1+t}{1-t}\right)^{n-1}\left(\cpk w+\frac{2nt}{(1-t)^2}\right). 
\end{align*}


In \cite{MR3810249}, Ira and Zhuang discussed shuffle-compatible permutation statistics in terms of shuffle algebras. In the Theorem 4.8, they proved that peak number $\pk$ is shuffle compatible and also characterized its shuffle algebra $\cA_{\pk}$. It turns out that the Proposition~\ref{p-orderpoly} gives another characterization of $\cA_{\pk}$. Before we state and prove the theorem, we review some definitions and results from~\cite{MR3810249}.

Suppose $\pi$ and $\sigma$ are two disjoint permutations of length $m$ and $n$ respectively. Then the {\em shuffle set} of $\pi$ and $\sigma$ is 
    \[
    \pi\shuffle\sigma=\{\tau: \text{$|\tau|=m+n$, $\pi$ and $\sigma$ are subsequences of $\tau$} \}.
    \]
A statistic $\st$ is {\em (linear) shuffle compatible} if for disjoint permutations $\pi$ and $\sigma$, the multiset $\{\{\,\st(\tau):\tau\in\pi\shuffle\sigma\, \}\}$ only depends on $\st(\pi)$, $\st(\sigma)$, $|\pi|$ and $|\sigma|$. 

Every linear permutation statistic $\st$ induces an equivalence relation on permutations. More precisely, two permutations $\pi$ and $\sigma$ are $\st$-equivalent if $\st(\pi)=\st(\sigma)$ and $|\pi|=|\sigma|$, and the {\em $\st$-equivalence class of $\pi$} is denoted by $[\pi]_{\st}$. Moreover, if $\st$ is shuffle compatible, one can associate to $\st$ a $\mathbb{Q}$-algebra as follows: first we associate to $\st$ a $\mathbb{Q}$-vector space by taking the $\st$-equivalence classes as a basis, then define multiplication by
\[
[\pi]_{\st}[\sigma]_{\st}=\sum_{\tau\in \pi\shuffle\sigma}[\tau]_{\st}.
\]
Here the shuffle-compatibility of $\st$ guarantees that the above multiplication is well-defined. In this case, we call the resulted algebra the {\em shuffle algebra of $\st$} and denote it by $\cA_{\st}$.

We recall the characterization of peak shuffle algebra $\cA_{\pk}$ from~\cite{MR3810249} is:
\begin{thm}[\cite{MR3810249}, Theorem 4.8 (b)]
The linear map on $\cA_{\pk}$ defined by 
    \[
    [\pi]_{\pk}\mapsto \begin{cases}\displaystyle\frac{2^{2\pk \pi+1}t^{\pk\pi+1}(1+t)^{|\pi|-2\pk\pi-1} }{ (1-t)^{|\pi|+1}}x^{|\pi|},& \text{if $|\pi|\geq 1$};\\[6pt]\displaystyle \frac{1}{1-t}, &\text{if $|\pi|=0$,}\end{cases}
    \]
    is a $\mathbb{Q}$-algebra isomorphism from $\cA_{\pk}$ to the span of 
    \[
    \left\{\frac{1}{1-t}\right\}\bigcup\left\{\frac{2^{2j+1}t^{j+1}(1+t)^{n-2j-1} }{(1-t)^{n+1}}x^{n} \right\}_{n\geq 1,\,0\leq j\leq\lfloor\frac{n-1}{2}\rfloor}
    \]
    a subalgebra of $\mathbb{Q}[[t\ast]][x]$.\qed
\end{thm}
Now we give another characterization of $\cA_{\pk}$ from our Proposition~\ref{p-orderpoly}.
\begin{thm}
The linear map on $\cA_{\pk}$ defined by
\[
[\pi]_{\pk}\mapsto 2^{2\pk \pi+1}\sum_{k=0}^{m-1-\pk \pi}\left(\binom{|\pi|+1}{k}\right)\binom{|\pi|-2\pk \pi-1}{m-1-\pk \pi-k}x^{|\pi|}
\]
is a $\mathbb{Q}$-algebra isomorphism from $\cA_{\pk}$ to the span of 
\[
\left\{1\right\}\bigcup\left\{2^{2j+1}\sum_{k=0}^{m-1-j}\left(\binom{n+1}{k}\right)\binom{n-2j-1}{m-1-j-k}x^{n} \right\}_{n\geq 1,\,0\leq j\leq\lfloor\frac{n-1}{2}\rfloor}
\]
a subalgebra of $\mathbb{Q}[x]^{\mathbb{N}}$, the algebra of functions $\mathbb{N}\to\mathbb{Q}[x]$ in the non-negative integer value $m$.
\end{thm}
\begin{proof}
Define a map $\kappa_m:\QSym\to \mathbb{Q}[x]$ first by
\[
\kappa_m(F_L)=K_{\Pk(L)}(1^m)x^n,
\]
and then linearly extend to the whole $\QSym$.
The following equation, \cite{S:epp} equation (3.1),
\[
K_{\Pk\pi}\cdot K_{\Pk\sigma}=\sum_{\tau\in\pi\shuffle\sigma}K_{\Pk\tau},
\]
implies that $\kappa_m$ is a $\mathbb{Q}$-algebra homomorphism.
So the map that takes $F_L$ to the function $\theta_L: m\mapsto \kappa_m(F_L)$ is a homomorphism from $\QSym$ to $\mathbb{Q}[x]^{\mathbb{N}}$. It follows from the Proposition~\ref{p-orderpoly} that
\[
\kappa_m(F_L)=2^{2\pk (L)+1}\sum_{k=0}^{m-1-\pk (L)}\left(\binom{n+1}{k}\right)\binom{n-2\pk (L)-1}{m-1-\pk (L)-k}x^n.
\]
Moreover, from  Theorem~\ref{t-opstem} one has
\[
\sum_{m=0}^{\infty}2^{2\pk (L)+1}\sum_{k=0}^{m-1-\pk (L)}\left(\binom{n+1}{k}\right)\binom{n-2\pk (L)-1}{m-1-\pk (L)-k}t^m=\frac{1}{2}\left(\frac{1+t}{1-t}\right)^{n+1}\left(\frac{4t}{(1+t)^2}\right)^{1+\pk(L)},
\]
which is the generating function of $\theta_L$ and only depends on $n$ and $\pk L$. They are clearly linearly independent for $L$ with distinct peak numbers. The result follows immediately from~\cite{MR3810249} Theorem 4.3.
\end{proof}


\newpage

\nocite{*}
\bibliographystyle{alpha}

\bibliography{ecpref}

\begin{thebibliography}{BHvW03}

\bibitem[ABS06]{MR2196760}
Marcelo Aguiar, Nantel Bergeron, and Frank Sottile.
\newblock Combinatorial {H}opf algebras and generalized {D}ehn-{S}ommerville
  relations.
\newblock {\em Compos. Math.}, 142(1):1--30, 2006.

\bibitem[AGRR20]{agrr:cqf}
Ron~M. Adin, Ira~M. Gessel, Victor Reiner, and Yuval Roichman.
\newblock Cyclic quasi-symmetric functions.
\newblock {\em S\'{e}m. Lothar. Combin.}, 82B:Art. 67, 12, 2020.

\bibitem[BHvW03]{MR1982883}
Louis~J. Billera, Samuel~K. Hsiao, and Stephanie van Willigenburg.
\newblock Peak quasisymmetric functions and {E}ulerian enumeration.
\newblock {\em Adv. Math.}, 176(2):248--276, 2003.

\bibitem[DMR16]{Dmr:tpo}
Mike Develin, Matthew Macauley, and Victor Reiner.
\newblock Toric partial orders.
\newblock {\em Trans. Amer. Math. Soc.}, 368(4):2263--2287, 2016.

\bibitem[Ges84]{MR777705}
Ira~M. Gessel.
\newblock Multipartite {$P$}-partitions and inner products of skew {S}chur
  functions.
\newblock In {\em Combinatorics and algebra ({B}oulder, {C}olo., 1983)},
  volume~34 of {\em Contemp. Math.}, pages 289--317. Amer. Math. Soc.,
  Providence, RI, 1984.

\bibitem[GZ18]{MR3810249}
Ira~M. Gessel and Yan Zhuang.
\newblock Shuffle-compatible permutation statistics.
\newblock {\em Adv. Math.}, 332:85--141, 2018.

\bibitem[Hiv00]{MR1794711}
Florent Hivert.
\newblock Hecke algebras, difference operators, and quasi-symmetric functions.
\newblock {\em Adv. Math.}, 155(2):181--238, 2000.

\bibitem[Sta72]{MR0332509}
Richard~P. Stanley.
\newblock {\em Ordered structures and partitions}.
\newblock Memoirs of the American Mathematical Society, No. 119. American
  Mathematical Society, Providence, R.I., 1972.

\bibitem[Sta01]{MR1897637}
Richard~P. Stanley.
\newblock Generalized riffle shuffles and quasisymmetric functions.
\newblock volume~5, pages 479--491. 2001.
\newblock Dedicated to the memory of Gian-Carlo Rota (Tianjin, 1999).

\bibitem[Ste97]{S:epp}
John~R. Stembridge.
\newblock Enriched {$P$}-partitions.
\newblock {\em Trans. Amer. Math. Soc.}, 349(2):763--788, 1997.

\bibitem[SW12]{MR3203651}
John Shareshian and Michelle~L. Wachs.
\newblock Chromatic quasisymmetric functions and {H}essenberg varieties.
\newblock In {\em Configuration spaces}, volume~14 of {\em CRM Series}, pages
  433--460. Ed. Norm., Pisa, 2012.

\end{thebibliography}

\end{document}